\documentclass[12pt]{article}
\usepackage{amsmath,amsthm,amssymb,times,titlesec}
\usepackage{graphicx}
\usepackage[T1]{fontenc}
\usepackage[utf8]{inputenc}
\usepackage{pdfsync}
\usepackage{float}
\usepackage{xcolor}
\usepackage[english]{babel}
\usepackage{verbatim}
\numberwithin{equation}{section}

\textheight 
            650pt
\textwidth 
           460pt

\titleformat{\subsubsection}
{\normalfont\fontsize{13}{14}\selectfont\itshape}{\thesubsubsection}{1em}{}

\newcommand\scaleddot{\scalebox{.89}{.}}
\makeatletter
\renewcommand{\ddddot}[1]{%
  {\mathop{\kern\z@#1}\limits^{\makebox[0pt][c]{\vbox to-2\ex@{\kern-\tw@\ex@\hbox{\normalfont\scaleddot\kern-0.5pt\scaleddot\kern-0.5pt\scaleddot\kern-0.5pt\scaleddot}\vss}}}}}

\newcommand{\abs}[1]{\lvert#1\rvert}
\newcommand\norm[1]{\left\lVert#1\right\rVert}

\newcommand{\FF}{{\mathcal F}}

\DeclareMathOperator{\sech}{sech}

\newtheorem{theorem}{Theorem}[section]
\newtheorem{lemma}[theorem]{Lemma}
\newtheorem{proposition}[theorem]{Proposition}

\newtheorem{remark}[theorem]{Remark}
\newtheorem{definition}[theorem]{Definition}
\theoremstyle{definition}

\title
{
	Solitary wave solutions of a Whitham-Boussinesq system
}

\author{E. Dinvay$^1$ and D. Nilsson$^2$}
\date{}

\begin{document}

\maketitle

\begin{center}
$^1$Department of Mathematics,
University of Bergen,
\\
Postbox 7800, 5020 Bergen, Norway.
\\
$^2$Department of Mathematical Sciences, NTNU,
\\
NO-7491 Trondheim, Norway.
\\
E-mail addresses:
Evgueni.Dinvay@uib.no;
Dag.Nilsson@ntnu.no
\end{center}

\begin{center}
{\it Current addresses:}
\\
$^1$Inria Rennes - Bretagne Atlantique,
\\
Campus universitaire de Beaulieu Avenue du G\'en\'eral Leclerc,
\\
35042 Rennes Cedex,
France.
\\
$^2$Department of Mathematics, Saarland University,
\\
Campus E24, 66123 Saarbr{\"u}cken, Germany.
\\
E-mail addresses:
Evgueni.Dinvay@inria.fr;
nilsson@math.uni-sb.de
\end{center}

\begin{abstract}
The travelling wave problem
for a particular bidirectional Whitham system
modelling surface water waves is under consideration.
This system firstly appeared in \cite{Dinvay_Dutykh_Kalisch},
where it was numerically shown to be stable and a good approximation
to the incompressible Euler equations.
In subsequent papers \cite{Dinvay, Dinvay_Tesfahun}
the initial-value problem was studied and
well-posedness in classical Sobolev spaces was proved.
Here we prove existence of solitary wave solutions
and provide their asymptotic description.
Our proof relies on a variational approach and
a concentration-compactness argument.
The main difficulties stem from the fact
that in the considered Euler-Lagrange equation
we have a non-local operator of positive order
appearing both in the linear and non-linear parts.
Our approach allows us to obtain solitary waves
for a particular Boussinesq system as well.
\end{abstract}

\section{Introduction}

\subsection{Motivation and background}

In this work we consider solitary wave solutions of the Whitham-Boussinesq system
\begin{align}
\label{wb-sys1}
	K \eta_t + K v_x + (\eta v)_x
	&= 0
	,\\
\label{wb-sys2}
	K v_t + \eta_x + vv_x
	&= 0
	,
\end{align}
where 
$K$ is a Fourier multiplier operator meaning
$\FF(Kf)(\xi) = k(\xi) \widehat{f}(\xi)$
for any tempered distribution $f \in \mathcal S'(\mathbb R)$.
Here $\mathcal{F}$ stands for the Fourier transform
\begin{equation*}
	\mathcal{F}(f)(\xi)
	=
	\widehat{f}(\xi)
	=
	\int_\mathbb{R}f(x)\mathrm{e}^{-\mathrm{i}x\xi}\ \mathrm{d}x
	.
\end{equation*}
A solitary wave is a solution of the form
\begin{equation}
\label{ansatz}
	\eta(x,t) = \eta(x+ct)
	,\
	v(x,t)=v(x+ct),
\end{equation}
with $\eta(x+ct), \ v(x+ct)\rightarrow 0$,
as $\abs{x+ct}\rightarrow \infty$.
%
%
%
%
The Fourier multipliers we will be considering in this paper include
\begin{equation}
\label{K_definition}
	K = \frac{D}{\tanh(D)}
\end{equation}
with $\mathrm{D}=-\mathrm{i}\partial_x$, and symbol
\(
	k(\xi) = \xi / \tanh \xi
	.
\)
Note that this operator is of order one, it is equivalent to the Bessel potential $J =(1-\partial_x^2)^{1/2}$
associated with the symbol
\(
	\langle \xi \rangle
	=
	\sqrt{1 + \xi^2}
,	
\)
 since
\(
	\xi / \tanh \xi \simeq	
	\langle \xi \rangle
	.
\)
Such choice of $K$ is motivated by the water wave problem,
when $\eta$ denotes the surface elevation
and $v$ is the fluid velocity at the surface. 
Another example is the operator $K = 1 - b \partial_x^2$, with the corresponding symbol $k(\xi) = 1 + b \xi^2$.  For this choice of $K$, \eqref{wb-sys1}--\eqref{wb-sys2} becomes a $(-b,b,0,b)$-Boussinesq system. We recall here that the general $(a, b, c, d)$-Boussinesq system is of the form
\begin{equation}
\label{Boussinesq_sys}
	\begin{aligned}
		\eta_t + v_x + (\eta v)_x
		+ av_{xxx} - b\eta_{xxt}
		&= 0
	,\\
		v_t + \eta_x + vv_x
		+ c\eta_{xxx} - dv_{xxt}
		&= 0
	,
	\end{aligned}
\end{equation}
and was derived in \cite{Bona_Chen_Saut1}.
It was shown in \cite{Chen_Nguyen_Sun2011}
that \eqref{Boussinesq_sys} exhibits
travelling waves for $a, c < 0$ and $b = d$.
The main result of the present paper (Theorem \ref{main-res-1}) implies that \eqref{Boussinesq_sys}
has solitary waves in the case
\(
	a = -b, c = 0, d = b
\)
for any $b > 0$.
To our knowledge this is a new result.
Note that with $b = 1/3$ this choice corresponds
to a physically relevant case
\cite{Bona_Chen_Saut1}.

The model \eqref{wb-sys1}-\eqref{wb-sys2}
with $K$ defined by \eqref{K_definition}
was formally derived in \cite{Dinvay_Dutykh_Kalisch}
from the incompressible Euler equations to model
fully dispersive shallow water waves.
In fact it can be regarded as a fully dispersive
improvement of the $(-1/3, 1/3, 0, 1/3)$-Boussinesq system.
System \eqref{wb-sys1}--\eqref{wb-sys2} was introduced
in \cite{Dinvay_Dutykh_Kalisch} as an extension
of the unidirectional Whitham equation
\begin{equation}
\label{Whitham}
	\partial_t \eta
	+ \sqrt{K^{-1}} \partial_x \eta
	+ \frac 32  \eta \partial_x \eta = 0
\end{equation}
allowing two-way wave propagation.
Equation \eqref{Whitham} was proved to be locally well-posed
in Sobolev spaces $H^s$ with $s > 3/2$ in \cite{Ehrnstrom_Escher_Pei}.
In recent years, several interesting phenomena predicted by Whitham
has been confirmed, for example, a solitary wave regime close to KdV \cite{Ehrnstrom_Groves_Wahlen},
the existence of a wave of greatest height
\cite{Ehrnstrom_Wahlen2019},
the existence of shocks \cite{Hur2017},
and modulational instability of 
steady periodic waves \cite{Hur_Johnson_2,Hur_Johnson}.
Apart from \eqref{wb-sys1}-\eqref{wb-sys2}
in recent years several 
bidirectional extensions of \eqref{Whitham} have been put forward to,
as for example
\begin{gather}
\begin{aligned}
	\eta_t
	&=
	- K^{-1}v_x-(\eta v)_x,\\
 v_t&=-\eta_x-vv_x
\end{aligned}
\label{asmp}
\end{gather}
and
\begin{gather}
\begin{aligned}
	\eta_t
	&=
	-v_x-(\eta v)_x,\\
	v_t
	&=-K^{-1}\eta_x-vv_x
	.
\end{aligned}
\label{hp}
\end{gather}
The first system \eqref{asmp} has a Hamiltonian structure
and was formally derived in \cite{Aceves_Sanchez_Minzoni_Panayotaros}
from
the incompressible Euler equations to model fully dispersive
shallow water waves whose propagation is allowed to be both left- and rightward, and appeared in
\cite{Lannes, Saut_Wang_Xu}
as a full dispersion system in the Boussinesq regime with the dispersion of the water
waves system.
There have been several investigations on this system:
local well-posedness \cite{Pei_Wang,Klein_Linares_Pilod_Saut}
(in homogeneous Sobolev spaces at a positive background), a
logarithmically cusped wave of greatest height
\cite{Ehrnstrom_Johnson_Claassen2019}.
In \cite{Pei_Wang}
they
impose an additional non-physical condition
$\eta \geqslant C > 0$.
Kalisch and Pilod \cite{Kalisch_Pilod}
have proved local well posedness
for a surface tension regularisation of System
\eqref{asmp}
with $( 1 - \beta \partial_x^2 ) K^{-1}$ standing instead
of $K^{-1}$ in the first equation.
So they managed to
remove the positivity assumption $\eta > 0$.
However, the maximal time of existence for their
regularisation is bounded by the capillary parameter $\beta > 0$.
The existence of solitary waves for this system was established in 
\cite{Nilsson_Wang}.
The second system \eqref{hp} was introduced in \cite{Hur_Pandey}
in order to better model modulational instabilities.
Indeed, it was found in \cite{Hur_Pandey} that,
when including the effects of surface tension,
the system \eqref{hp} gives accurate predicitions of
Benjamin-Feir instabilities.
It was pointed out in \cite{Carter} that System \eqref{hp}
has also a Hamiltonian structure.
To the authors knowledge
neither well-posedness nor the existence of solitary waves
for the system \eqref{hp} have been established yet.
There are also numerical results on
the validity of the both systems \eqref{asmp}, \eqref{hp}
for modelling waves on shallow water \cite{Carter},
numerical bifurcation and spectral stability \cite{Claassen_Johnson}.

System \eqref{wb-sys1}--\eqref{wb-sys2}
has been recently shown to be well-posed in
$H^s(\mathbb R) \times H^{s+1/2}(\mathbb R)$
with $s > -1/10$
in \cite{Dinvay, Dinvay_Tesfahun}.
Moreover, the result is global
for $s \geqslant 0$
if the initial data $(\eta_0, v_0)$
has sufficiently small
$L^2 \times H^{1/2}$-norm.
The latter is the main advantage of 
Equations \eqref{wb-sys1}--\eqref{wb-sys2} comparing
with the other models \eqref{asmp}, \eqref{hp}.
%
%

There is a fully-dispersive Green-Naghdi type model introduced by Duch\^ene, Israwi and Talhouk
\cite{Duchene_Israwi_Talhouk} that was not considered in \cite{Dinvay_Dutykh_Kalisch}.
Existence of solitary wave solutions for this system
was established in \cite{Duchene_Nilsson_Wahlen}.
For more discussion on the Cauchy problem
and rigorous justification of the various Whitham related equations
we refer to \cite{Klein_Linares_Pilod_Saut}.

The main aim of the current paper is to prove the existence
of solitary wave solutions for \eqref{wb-sys1}-\eqref{wb-sys2}
with $K$ being an admissible Fourier multiplier,
see Definition \ref{admissible} below. Both $K$ given by \eqref{K_definition} and $K = 1 - b \partial_x^2$ are examples of admissible Fourier multipliers.
Note that the existence of solitary waves supports
validity of System \eqref{wb-sys1}-\eqref{wb-sys2}
from physical perspective
as a weakly nonlinear wave model.
We use a variational approach together with Lion's method of
concentration-compactness \cite{Lions}
to establish the existence of
solitary wave solutions of \eqref{wb-sys1}--\eqref{wb-sys2}.
This approach has been used extensively to prove
existence of solitary wave solutions to equations of the form 
\begin{equation}\label{model-eq-example}
u_t+Lu_x+n(u)_x=0,
\end{equation}
where $L$ is a Fourier multiplier operator of order $s$,
essentially meaning that the symbol of the operator can be bounded from above and below
by $\abs{\xi}^s$ up to a constant,
and $n(u)$ is a homogeneous nonlinear term.
Under the travelling wave ansatz $u = u(x + ct)$,
equation \eqref{model-eq-example} becomes
\begin{equation}
\label{travel-intro}
	cu+Lu+n(u)=0.
\end{equation}
In \cite{Weinstein1987} existence and stability of solitary wave solutions for long wave model equations of the form \eqref{model-eq-example}, with $s\geq 1$, was established. This approach was later used in \cite{Albert_Bona_Saut} to prove existence of solitary waves for an equation used to model stratified fluids, with $s=1$,  and was later generalized in \cite{Albert} to $s\geq 1$.
A class of Whitham type equations of the form \eqref{model-eq-example} was studied in \cite{Ehrnstrom_Groves_Wahlen},
with a Fourier multiplier operator of negative order.
In this case the resulting functional in the constrained minimization problem is not coercive. This makes the application of the concentration compactness theorem a lot more technical, requiring the authors to use a strategy developed in \cite{Buffoni2004, Groves_Wahlen2011}
and first consider a related penalized functional acting on periodic functions.
In the recent work \cite{Stefanov_Wright2018} an entirely different
approach to proving the existence of solitary wave solutions
of the Whitham equation, based on the implicit function theorem
instead, was presented.
Arnesen proved existence of solitary wave solutions to
two different classes of model equations \cite{Arnesen},
one of them of the form \eqref{model-eq-example}, for $s>0$.
Results, similar and previous to those of Arnesen,
were obtained in \cite{Linares_Pilod_Saut2015}
in application to two particular cases, namely,
the fractional Korteweg-de Vries
and the fractional Benjamin-Bona-Mahony equations.
The case when the nonlinearity $n$ is allowed to be inhomogeneous was considered in \cite{Maehlen},
where the author proved the existence of solitary wave solutions of \eqref{model-eq-example}, for operators of positive order and with weak assumptions on the regularity of the symbol.

These methods have also been applied to bidirectional
Whitham type equations.
As mentioned above, in \cite{Duchene_Nilsson_Wahlen} the authors established the existence of solitary waves for the class of modified Green--Naghdi equations introduced in
\cite{Duchene_Israwi_Talhouk},
and in \cite{Nilsson_Wang}
the authors proved the existence of solitary waves
for \eqref{asmp}.
Just as in \cite{Ehrnstrom_Groves_Wahlen},
both of the functionals appearing in
\cite{Duchene_Nilsson_Wahlen, Nilsson_Wang}
are noncoercive, so the minimization arguments adapted to noncoercive functionals developed in \cite{Buffoni2004, Groves_Wahlen2011}
are used in order to obtain the existence of minimizers. In addition, the Fourier multiplier operator is entangled with the nonlinearity in \cite{Duchene_Nilsson_Wahlen, Nilsson_Wang}, which makes the proofs more technical.
%
%
\subsection{The minimization problem}
\label{minproblem} 
%
%
We formulate the problem in the variational settings.
A Hamiltonian structure \cite{Dinvay}
of System \eqref{wb-sys1}--\eqref{wb-sys2}
allows us to do this in a relatively straightforward way.
Indeed, under the travelling wave ansatz \eqref{ansatz},
equations \eqref{wb-sys1}--\eqref{wb-sys2} can be written as
\begin{align}\label{wb-sys-travelling1}
Kv+\eta v+cK\eta&=0,\\
\eta+\frac{v^2}{2}+cKv&=0,\label{wb-sys-travelling2}
\end{align}
where the constants of integration are set to zero since we are considering solitary wave solutions.

Regarding the Hamiltonian and momentum
\begin{align*}
\mathcal{H}(\eta,v)&=\frac{1}{2}\int_{\mathbb{R}}\eta^2+vKv+\eta v^2\ \mathrm{d}x,\\
\mathcal{I}(\eta,v)&=\int_\mathbb{R}\eta Kv\ \mathrm{d}x,
\end{align*}
one can notice that Equation
\eqref{wb-sys-travelling1} can be written as
\begin{equation*}
\mathrm{d}_v\mathcal{H}+c\mathrm{d}_v\mathcal{I}=0,
\end{equation*}
and Equation \eqref{wb-sys-travelling2} as 
\begin{equation*}
\mathrm{d}_\eta\mathcal{H}+c\mathrm{d}_\eta\mathcal{I}=0.
\end{equation*}
One can try to proceed further with this formulation
as was done for the
$(a, b, c, d)$-Boussinesq system \eqref{Boussinesq_sys}
in \cite{Chen_Nguyen_Sun2011}.
However, instead of looking for critical points
of the functional
\(
	\mathcal H + c \mathcal I
\)
we reduce System
\eqref{wb-sys-travelling1}-\eqref{wb-sys-travelling2} to
a single travelling wave equation
that can in turn be interpreted as a constrained minimization problem.
Note that our approach allows us to extend the results
obtained in \cite{Chen_Nguyen_Sun2011}.
We can derive a travelling wave equation in the following way. 
In \eqref{wb-sys-travelling1}--\eqref{wb-sys-travelling2} we make the change of variable $v=K^{-1/2}\tilde{v}$, which yields the new system
\begin{align}
\label{trav1}
K^{1/2}\tilde{v}+\eta(K^{-1/2}\tilde{v})+cK\eta&=0,\\
\eta+\frac{(K^{-1/2}\tilde{v})^2}{2}+cK^{1/2}\tilde{v}&=0
\label{trav2}
.
\end{align}
From \eqref{trav2} we get that
\begin{equation}\label{eta-solution}
\eta=-\frac{(K^{-1/2}\tilde{v})^2}{2}-cK^{1/2}\tilde{v},
\end{equation}
and inserting this into \eqref{trav1} yields
\begin{equation}\label{maineq0}
\tilde{v}-K^{-1/2}\left(\frac{(K^{-1/2}\tilde{v})^3}{2}\right)-cK^{-1/2}\big((K^{1/2}\tilde{v})(K^{-1/2}\tilde{v})\big)-cK^{1/2}\left(\frac{(K^{-1/2}\tilde{v})^2}{2}\right)-c^2K\tilde{v}=0.
\end{equation}
Here we make the change of variables $\tilde{v}=cu$ so that \eqref{maineq0} becomes
\begin{equation}
\label{maineq1}
\frac{1}{c^2}u-K^{-1/2}\left(\frac{(K^{-1/2}u)^3}{2}\right)-K^{-1/2}\big((K^\frac{1}{2}u)(K^{-1/2}u)\big)-K^{1/2}\left(\frac{(K^{-1/2}u)^2}{2}\right)-Ku=0.
\end{equation}
Now let us show that Equation \eqref{maineq1} represents
an Euler-Lagrange equation for some functional.
Indeed, regard the surface elevation and velocity
defined by $u$ as follows
\begin{align}
\label{eta-u}
	\eta_u
	&=
	-c^2\left( \frac{(K^{-1/2}u)^2}{2} + K^{1/2}u \right)
	, \\
\label{v-u}
	v_u
	&=
	cK^{-1/2}u
	,
\end{align}
and note that
\begin{equation*}
\mathcal{H}(\eta_u,v_u)+c\mathcal{I}(\eta_u,v_u)=c^4\left[-\frac{1}{2}\int_\mathbb{R} uKu+(K^{1/2}u)(K^{-1/2}u)^2+\frac{(K^{-1/2}u)^4}{4}\ \mathrm{d}x+\frac{1}{2c^2}\int_\mathbb{R} u^2\ \mathrm{d}x\right]
,
\end{equation*}
which leads us to define
\begin{align*}
\mathcal{E}(u)&=\frac{1}{2}\int_\mathbb{R}uKu+(K^{1/2}u)(K^{-1/2}u)^2+\frac{(K^{-1/2}u)^4}{4}\ \mathrm{d}x,\\
&=\frac 12 \int_\mathbb{R}
	\left(K^{1/2}u + \frac 12 (K^{-1/2}u)^2\right)^2
	\ \mathrm{d}x\\
\mathcal{Q}(u)&=\frac{1}{2}\int_\mathbb{R}u^2\ \mathrm{d}x.
\end{align*}
We then note that equation \eqref{maineq1} can be written as
\begin{equation*}
\mathrm{d}\mathcal{E}(u)+\lambda\mathrm{d}\mathcal{Q}(u) = 0,
\end{equation*}
where $\lambda=-1/c^2$.
Hence, in order to find solutions of \eqref{maineq1} we can consider the constrained minimization problem
\begin{equation}
\label{min2}
	\inf _{u \in U_q} \mathcal E(u)
	\quad
	\mbox{ with }
	\quad
	U_q =
	\left\{
		u \in H^{1/2}(\mathbb R) \ \colon \mathcal Q(u) = q
	\right\}
	.
\end{equation}

Instead of working with the specific Fourier multiplier $K$,
we will work with a more general class of Fourier multipliers,
and thus a more general constrained minimization problem.
The proof does not get much more complicated
if we consider a general class of multipliers.
Moreover, as was mentioned in the introduction,
it allows us to treat also
a Boussinesq system not considered from this perspective before.

\begin{definition}[Admissible Fourier multipliers]
\label{def-admissible}
\label{admissible}
Let operator $L$ be a Fourier multiplier,
with symbol $m$, i.e.
\begin{equation*}
\mathcal F(Lf)(\xi) = m(\xi) \widehat{f}(\xi).
\end{equation*}
We say that $L$ is admissible
if $m$ is even, $m(0)>0$ and for some $s'>1$ and $s>1/2$
the symbol satisfies the following restrictions.
\begin{enumerate}
\item
The function $\xi\mapsto \frac{m(\xi)}{\langle \xi\rangle^s}$
is uniformly continuous, and
\begin{align}
0<m(\xi)-m(0) & \lesssim \abs{\xi}^{s'}
\, \mbox{ for } \,
\abs{\xi} \leqslant 1
,
\label{m1}\\
m(\xi)-m(0) & \simeq \abs{\xi}^s
\, \mbox{ for } \,
\abs{\xi}>1\label{m2}
.
\end{align}
\item
For each $\varepsilon> 0$ the kernel
of operator $L^{-1/2}$ satisfies
\begin{equation}
\label{int-cond}
	\mathcal F^{-1} \left(m^{-1/2}\right)
	\in L^2(\mathbb{R} \setminus (- \varepsilon, \varepsilon)).
\end{equation}
There exists
\(
	p \in (1, 2) \cap [2 / (s+1),2)
\)
such that
\begin{equation}
\label{int-cond2}
	\mathcal F^{-1} \left(m^{-1/2}\right)
	\in L^p(- 1, 1)
	.
\end{equation}
\end{enumerate}
\end{definition}

The symbol $m(\xi)=\xi/\tanh(\xi)$ satisfies the conditions
of Definition \ref{admissible} with $s = 1$ and $s'= 2$
as was shown in
\cite{Ehrnstrom_Wahlen2019}.
For the symbol $m(\xi)= 1 + b \xi^2$ with $b > 0$
we have $s = s' = 2$ and
\(
	m^{-1/2} \in L^2(\mathbb R)
	.
\)
In particular,
\(
	\mathcal F^{-1} \left(m^{-1/2}\right)
	\in L^2(\mathbb R)
\)
and so \eqref{int-cond}, \eqref{int-cond2} hold.

We have the corresponding functional
\begin{equation}
\label{general_E_definition}
	\mathcal E(u)=\frac 12 \int_\mathbb{R}
	\left(L^{1/2}u + \frac 12 (L^{-1/2}u)^2\right)^2
	\ \mathrm{d}x
\end{equation}
%
%
defined on $H^{s/2}(\mathbb{R})$.
Our main goal is then to obtain a solution of the minimization problem
\begin{equation}
\label{general_min_problem}
	\inf _{u \in U_q} \mathcal E(u)
	\quad
	\mbox{ with }
	\quad
	U_q =
	\left\{
		u \in H^{s/2}(\mathbb R) \ \colon \mathcal Q(u) = q
	\right\}
	.
\end{equation}
For convenience we separate $\mathcal E$ into the functionals
\begin{align*}
\mathcal{L}(u)&=\frac{1}{2}\int_\mathbb{R}uLu\ \mathrm{d}x,\\
\mathcal{N}_c(u)&=\frac{1}{2}\int_\mathbb{R}L^{1/2}u(L^{-1/2}u)^2\ \mathrm{d}x,\\
\mathcal{N}_r(u)&=\frac{1}{2}\int_\mathbb{R}\frac{(L^{-1/2}u)^4}{4}\ \mathrm{d}x
\end{align*}
so that
\begin{equation*}
	\mathcal{E}(u) = \mathcal{L}(u) + \mathcal{N}(u)
\end{equation*}
where
\[
	\mathcal N(u) = \mathcal{N}_c(u) + \mathcal{N}_r(u)
	.
\]

We are now ready to state our main results.

\begin{theorem}\label{main-res-1}
Let $D_q$ be the set of minimizers of $\mathcal{E}$ over $U_q$. There exists $q_0>0$ such that for each $q\in (0,q_0)$, the set $D_q$ is nonempty and $\norm{u}_{H^\frac{s}{2}}^2\lesssim q$ for $u\in D_q$.
Each element of $D_q$ is a solution of the Euler--Lagrange equation
\begin{equation}\label{travelling-L}
\lambda u+L^{-1/2}\left(\frac{(L^{-1/2}u)^3}{2}\right)+L^{-1/2}\big((L^{1/2}u)(L^{-1/2}u)\big)+L^{1/2}\left(\frac{(L^{-1/2}u)^2}{2}\right)+Lu=0.
\end{equation}
The Lagrange multiplier $\lambda$ satisfies 
\begin{equation}\label{crude-est-lambda}
\frac{m(0)}{2}< -\lambda< m(0)-Dq^\beta,
\end{equation}
where $\beta=\frac{s'}{2s'-1}$ and $D$ is a positive constant.
\end{theorem}
Here and throughout the paper we write $f\lesssim g$, when $\abs{\frac{f}{g}}$ is uniformly bounded from above, and $f\simeq g$ when $f\lesssim g\lesssim f$. Our other main result concerns the asymptotic behavior of  travelling wave solutions of \eqref{wb-sys1}--\eqref{wb-sys2}.

\begin{theorem}
\label{main-res-2}
If $L=K$ defined by \eqref{K_definition}
then there exists $q_0>0$ such that for any
$q\in(0,q_0)$ each minimizer $u\in D_q$ belongs to $H^{r}(\mathbb{R})$
for any $r \geqslant 0$ with $\norm{u}_{H^{r}}^2\lesssim q$,
and moreover, it satisfies the following long wave asymptotics
\[
\sup_{u\in D_q}\inf_{x_0\in \mathbb{R}}
\norm{
	q^{-2/3}u(q^{-1/3}\cdot)-\psi_{\text{KdV}}(\cdot-x_0)
}	_{H^1(\mathbb{R})}\lesssim q^{1/6}
,
\]
whereas the corresponding surface elevation \eqref{eta-u}
and speed \eqref{v-u} satisfy
\begin{align*}
\sup_{u\in D_q}\inf_{x_0\in \mathbb{R}}
\norm{
	q^{-2/3}\eta_u(q^{-1/3}\cdot)
	+
	\psi_{\text{KdV}}(\cdot-x_0)
}_{H^{1/2}(\mathbb{R})}
\lesssim q^{1/6}
,\\
\sup_{u\in D_q}\inf_{x_0\in \mathbb{R}}
\norm{
	q^{-2/3} v_u(q^{-1/3}\cdot)
	+
	\psi_{\text{KdV}}(\cdot-x_0)
}_{H^{3/2}(\mathbb{R})}
\lesssim q^{1/6}
,
\end{align*}
where 
\begin{equation*}
\psi_{\text{KdV}}(x)=-\lambda_0\sech^2\left(\frac{1}{2}\sqrt{3\lambda_0}x\right)
\end{equation*}
and $\lambda_0=3/\sqrt[3]{16}$. In addition, the Lagrange multiplier $\lambda$ satisfies
\begin{equation*}
\lambda=-1+\lambda_0q^{2/3}+\mathcal{O}(q^{5/6}).
\end{equation*}
\end{theorem}
 We discuss here briefly how to prove
Theorems \ref{main-res-1}, \ref{main-res-2}. 
The main ingredient in proving Theorem \ref{main-res-1} is Lion's concentration compactness theorem \cite{Lions}:
\begin{theorem}[Concentration-compactness]
\label{T.concentration-compactness}
Any sequence $\{e_n\}_{n\in\mathbb{N}}\subset L^1(\mathbb{R})$ of non-negative functions such that
\[\lim_{n\to \infty}\int_{\mathbb{R}}e_n\ \mathrm{d} x= I>0\]
admits a subsequence, denoted again $\{e_n\}_{n\in\mathbb{N}}$, for which one of the following phenomena occurs.
\begin{itemize}
\item (Vanishing) For each $r>0$, one has
\[\lim_{n\to\infty} \left( \sup_{x\in\mathbb{R}}\int_{x-r}^{x+r}e_n\ \mathrm{d} x\right)=0.\]
\item (Dichotomy) There are real sequences $\{x_n\}_{n\in\mathbb{N}},\{M_n\}_{n\in\mathbb{N}},\{N_n\}_{n\in\mathbb{N}}\subset\mathbb{R}$ and $I^* \in(0,I)$ such that $M_n,N_n\to\infty$, $M_n/N_n\to 0$, and
\[\int_{x_n-M_n}^{x_n+M_n}e_n\ \mathrm{d} x\to I^* \quad \text{ and } \quad  \int_{x_n-N_n}^{x_n+N_n}e_n\ \mathrm{d} x\to I^*,  \]
as $n\to\infty$.
\item (Concentration) There exists a sequence $\{x_n\}_{n\in\mathbb{N}}\subset\mathbb{R}$ with the property that for each $\epsilon>0$, there exists $r>0$ with 
\[\int_{x_n-r}^{x_n+r} e_n\ \mathbb{d} x \geq I-\epsilon,\]
for all $n\in\mathbb{N}$.
\end{itemize}
\end{theorem}
We will apply this theorem to $e_n=u_n^2$, where $\{u_n\}_{n=1}^\infty$ is a minimizing sequence, and show that the vanishing and dichotomy scenarios cannot occur.
Then we obtain a convergent subsequence of $\{u_n\}_{n=1}^\infty$
using the concentration scenario.
The functional $\mathcal{E}$ is similar to the corresponding functionals appearing in \cite{Duchene_Nilsson_Wahlen, Nilsson_Wang}, in the sense that the Fourier multiplier and the nonlinerity are entangled.
However, in contrast with \cite{Duchene_Nilsson_Wahlen, Nilsson_Wang},
our functional $\mathcal{E}$ is 
bounded from below,
hence the penalization argument
of  \cite{Buffoni2004, Groves_Wahlen2011}
is not necessary in our case.

In \cite{Duchene_Nilsson_Wahlen} the exclusion of dichotomy
gets more technical due to the entanglement of
the Fourier multiplier with the nonlinearity,
and this is true for the present work as well.
In contrast, the exclusion of the vanishing scenario is straightforward
in \cite{Duchene_Nilsson_Wahlen},
while this is not the case in the present work.
This is due to the fact that in \cite{Duchene_Nilsson_Wahlen}
the constrained minimization problem is formulated in $H^s(\mathbb{R})$,
$s> 1/2$, allowing the use of the embedding
$H^s(\mathbb{R})\hookrightarrow L^\infty(\mathbb{R})$,
while our problem is formulated in $H^{s/2}(\mathbb{R})$,
preventing the use of this embedding.
Instead we show that if $\{u_n\}_{n=1}^\infty$ is vanishing,
then $L^{-1/2}u_n$ is vanishing as well,
which leads to a contradiction.
In order to show that $L^{-1/2}u_n$ is vanishing we make use
of the integrability assumptions \eqref{int-cond}, \eqref{int-cond2}
imposed on the kernel of $L$,
and this is the only instance where these assumptions are used.
Our other assumptions on $L$ as are similar to those in \cite{Maehlen},
and we are able to adopt many of the methods used in that paper
to our present work.
We give here a brief explanation on where the remaining assumptions
on $m$ in Definition \ref{def-admissible} are used.
The uniform continuity of $m(\xi)/\langle \xi\rangle^s$
is used to prove Lemma \ref{uniform-cont},
which in turn is used to exclude the dichotomy scenario.
The upper bound in \eqref{m1} is used to prove
Proposition \ref{near-min},
which allows us to define near minimizers.
In Proposition \ref{near-min} we also make use of the lower bound
on $s'$, which tells us that the number $\beta = s'/(2s'-1)$
is strictly less than $1$.
The latter and
the lower bound in \eqref{m1} are used to prove subadditivity
of $I_q$ in Proposition \ref{subadditivity}.
The lower bound in \eqref{m2} essentially gives us coercivity
and enables us to obtain a uniform upper bound on near minimizers
in $H^{s/2}$-norm, in Proposition \ref{near-min-est-prop}.
The lower bound on $s$ is needed when excluding
the vanishing scenario in Proposition \ref{vanishing-prop}.
Here we make use of the fact that $L^{-1/2}u \in H^s$,
for $u\in H^{s/2}$, and since $s>1/2$ we can conclude
that $L^{-1/2}u \in L^\infty$.
Finally, the upper bound in \eqref{m2} is essentially used to estimate
$\norm{L^{1/2}u_n}_{L^2}\lesssim \norm{u_n}_{H^{s/2}}$,
where $u_n$ is a minimizing sequence,
and we know from before that $u_n$ is uniformly bounded
in $H^{s/2}$-norm.
Hence $\norm{L^{1/2}u_n}_{L^2}$ is also uniformly bounded,
and this fact is used in Proposition \ref{D-prop} to exclude dichotomy,
and also in Proposition \ref{C-prop} while
proving existence of a minimizer.

Theorem \ref{main-res-2} is established using more standard arguments,
see for example \cite{Duchene_Nilsson_Wahlen, Ehrnstrom_Groves_Wahlen}.

\section{Technical results}\label{technical-results}

The current section is devoted to the general properties
of the functionals introduced above.
We start with a useful proposition on continuity
of symbol $m(\xi)$ described by Definition \ref{admissible}.
\begin{lemma}\label{uniform-cont}
There is a function $\omega\ :\ \mathbb{R}\rightarrow [0,\infty)$, bounded from above by a polynomial, with $\lim_{\lambda\rightarrow 0}\omega(\lambda)=0$, such that
\begin{equation*}
\abs{m(\xi)-m(\eta)}\leq \omega(\xi-\eta)\langle\xi\rangle^\frac{s}{2}\langle \eta\rangle^\frac{s}{2}.
\end{equation*} 
\end{lemma}
\begin{proof}
The proof is given in \cite[Proposition 2.1]{Maehlen}.
\end{proof}

The following functional estimates will be used a lot
in the text below,
sometimes without references.

\begin{proposition}\label{L-est}
For any $u\in H^{s/2}(\mathbb{R})$ one has
\begin{equation*}
\mathcal{L}(u)\simeq \norm{u}_{H^{s/2}}^2.
\end{equation*}
\end{proposition}
\begin{proof}
This is immediate from Definition \ref{admissible}.
\end{proof}
\begin{proposition}\label{N-est}
For any $s>1/2$ and $u\in H^{s/2}(\mathbb{R})$ one has
\begin{align}
\abs{\mathcal{N}_c(u)}&\lesssim\norm{u}_{L^2}^2\norm{u}_{H^{s/2}},\label{N_c-1}\\
\abs{\mathcal{N}_r(u)}&\lesssim \norm{u}_{L^2}^4.\label{N_r-1}
\end{align}
\end{proposition}
\begin{proof}
Inequality \eqref{N_r-1} follows from the Sobolev embedding
\begin{equation*}
\abs{\mathcal{N}_r(u)}=\frac{1}{8}\norm{L^{-1/2}u}_{L^4}^4\lesssim \norm{\abs{\partial_x}^{1/4}L^{-1/2}u}_{L^2}^4\lesssim \norm{J^{1/4-s/2}u}_{L^2}^4\lesssim\norm{u}_{L^2}^4.
\end{equation*}
Inequality \eqref{N_c-1} follows from \eqref{N_r-1} and Hölder's inequality.
\end{proof}
\begin{proposition}\label{diff-est}
For $s>1/2$ and $u,h\in H^\frac{s}{2}(\mathbb{R})$
the Fr\'echet derivative of $\mathcal{E}$ satisfies
\begin{equation*}
\abs{\mathrm{d}\mathcal{E}(u)(h)}\lesssim \norm{u}_{H^\frac{s}{2}}(1+\norm{u}_{L^2}+\norm{u}_{L^2}^2)\norm{h}_{H^\frac{s}{2}}
\end{equation*}
\end{proposition}
\begin{proof}
We first note that
\begin{equation*}
\abs{\mathrm{d}\mathcal{L}(u)(h)}\lesssim \norm{u}_{H^\frac{s}{2}}\norm{h}_{H^\frac{s}{2}}.
\end{equation*}
Next consider
\begin{equation}\label{dN_c}
\mathrm{d}\mathcal{N}_c(u)(h)=\frac{1}{2}\int_\mathbb{R}L^{1/2}h(L^{-1/2}u)^2+2uL^{1/2}(L^{-1/2}uL^{-1/2}h)\ \mathrm{d}x,
\end{equation}
where
\begin{equation*}
\norm{L^{1/2}h(L^{-1/2}u)^2}_{L^1}\leq\norm{L^{1/2}h}_{L^2}\norm{L^{-1/2}u}_{L^4}^2\lesssim \norm{u}_{L^2}^2\norm{h}_{H^\frac{s}{2}},
\end{equation*}
\begin{align*}
\norm{uL^{1/2}(L^{-1/2}uL^{-1/2}h)}_{L^1}&\leq \norm{u}_{L^2}\norm{L^{1/2}(L^{-1/2}uL^{-1/2}h)}_{L^2}\\
&\lesssim \norm{u}_{L^2}\norm{L^{-1/2}uL^{-1/2}h}_{H^\frac{s}{2}}\\
&\lesssim \norm{u}_{L^2}\norm{L^{-1/2}u}_{H^\frac{s}{2}}\norm{L^{-1/2}h}_{H^s}\\
&\lesssim\norm{u}_{L^2}^2\norm{h}_{H^\frac{s}{2}}.
\end{align*}
Using the above estimates in \eqref{dN_c}, we immediately get that
\begin{equation*}
\abs{\mathrm{d}\mathcal{N}_c(u)(h)}\lesssim \norm{u}_{L^2}^2\norm{h}_{H^\frac{s}{2}}.
\end{equation*}
In a similar way we find that
\begin{equation*}
\abs{\mathrm{d}\mathcal{N}_r(u)(h)}\lesssim\norm{u}_{L^2}^3\norm{h}_{H^\frac{s}{2}},
\end{equation*}
which concludes the proof.
\end{proof}
We next record a decomposition result for $\mathcal{N}_c$.
\begin{lemma}\label{N-decomp}
Let $u\in H^{s/2}(\mathbb{R})$.
Then
\begin{equation*}
\mathcal{N}_c(u)=\frac{1}{2\sqrt{m(0)}}\int_{\mathbb{R}}u^3\ \mathrm{d}x+\mathcal{N}_{c1}(u)+\mathcal{N}_{c2}(u)+\mathcal{N}_{c3}(u),
\end{equation*}
where 
\begin{align*}
\mathcal{N}_{1c}(u)&=\frac{\sqrt{m(0)}}{2}\int_\mathbb{R}u\left((L^{-1/2}-m^{-1/2}(0))u\right)^2\ \mathrm{d}x\\
\mathcal{N}_{2c}(u)&=\int_\mathbb{R}u^2(L^{-1/2}-m^{-1/2}(0))u\ \mathrm{d}x\\
\mathcal{N}_{3c}(u)&=\frac{1}{2}\int_\mathbb{R}(L^{-1/2}u)^2(L^{1/2}-m^{1/2}(0))u\ \mathrm{d}x,
\end{align*}
and
\begin{align*}
\abs{\mathcal{N}_{2c}(u)}&\leq \norm{u}_{L^4}^2\norm{(L^{-1/2}-m^{-1/2}(0))u}_{L^2}
,\\
\abs{\mathcal{N}_{3c}(u)}&\lesssim \norm{u}_{L^2}^2\norm{(L^{1/2}-m^{1/2}(0))u}_{L^2}.
\end{align*}
\end{lemma}
\begin{proof}
The proof is straightforward and is therefore omitted.
\end{proof}

Before we continue we want to make a remark on the convolution theorem.
According to our choice of the Fourier transform normalisation,
for any two functions $f$ and $g$ we have
\[
	\mathcal F (fg)
	= 
	\frac 1{2\pi} \widehat f * \widehat g
\]
where star stands for convolution.

\begin{lemma}
\label{translation_invariance}
	The functional $\mathcal E$ defined by \eqref{general_E_definition}
	is translation invariant.
	In other words,
	for any $u \in H^{s/2}(\mathbb R)$
	then
	\(
		\mathcal E(u_h) = \mathcal E(u)
		,
	\)
where $u_h(x) = u(x - h)$ denotes translation by $h \in \mathbb R$.
\end{lemma}
\begin{proof}
Due to the property $\widehat{u_h}(\xi) = e^{-\mathrm{i}h\xi} \widehat{u}(\xi)$
and the Plancherel theorem we have
\begin{multline*}
	\mathcal E(u_h) =
	\frac 1{4\pi} \int_\mathbb{R}
	\left|
		\sqrt{m(\xi)} \widehat{u_h}(\xi)
		+ \frac 12 \mathcal F \left(\left( L^{-1/2}u_h \right)^2\right) (\xi)
	\right|^2
	\ \mathrm{d}\xi
	\\
	=	
	\frac 1{4\pi} \int_\mathbb{R}
	\left|
		e^{-\mathrm{i}h\xi} \sqrt{m(\xi)} \widehat{u}(\xi)
		+ \frac 12 e^{-\mathrm{i}h\xi} \mathcal F \left(\left( L^{-1/2}u \right)^2\right) (\xi)
	\right|^2
	\ \mathrm{d}\xi
	= \mathcal E(u)
\end{multline*}
where we have also used the fact that the Fourier transform
of multiplication is convolution
of Fourier transforms
up to a normalization constant.
%
%
\end{proof}

In the following lemma we provide a slightly sharper
estimate for $\mathcal N_c$.
It will be the first step
towards the non-vanishing proof given below.
\begin{lemma}
\label{Nc-estimate}
	For $s > 1/2$ the following estimate hold true	
	\[
		| \mathcal N_c(u) | \lesssim
		\lVert u \rVert_{L^2 (\mathbb R)}^2
		\lVert L^{-1/2}u \rVert_{L^{\infty} (\mathbb R)}
	\]
\end{lemma}
\begin{proof}
	Clearly, $L^{-1/2}u \in L^{\infty} (\mathbb R)$
	and so applying a Kato--Ponce type estimate obtain
	\begin{multline*}
		| \mathcal N_c(u) |
		\lesssim
		\lVert u \rVert_{L^2}
		\lVert L^{1/2}( L^{-1/2}u )^2 \rVert_{L^2}
		\lesssim
		\lVert u \rVert_{L^2}
		\lVert J^{s/2}( L^{-1/2}u )^2 \rVert_{L^2}
		\\		
		\lesssim
		\lVert u \rVert_{L^2}
		\lVert J^{s/2} L^{-1/2}u \rVert_{L^2}
		\lVert L^{-1/2}u \rVert_{L^{\infty}}
		\lesssim
		\lVert u \rVert_{L^2}^2
		\lVert L^{-1/2}u \rVert_{L^{\infty}}
		.
	\end{multline*}
\end{proof}
We finish this section with a lemma which will be used when ruling out the dichotomy scenario.

\begin{lemma}\label{commutator}
Let $\varphi\in \mathcal{S}(\mathbb{R})$, and let $A_r\colon H^{s/2}(\mathbb{R})\to H^{s/2}(\mathbb{R})$, $B_r\colon L^2(\mathbb{R})\to L^2(\mathbb{R})$ be the operators
\begin{align*}
A_rf&=[L,\varphi\left(\frac{.}{r}\right)]f,\\
B_rf&=[L^{-\frac{1}{2}},\varphi\left(\frac{.}{r}\right)]f.
\end{align*}
Then the operator norms
\begin{equation*}
	\norm{A_r}   
	,
	\norm{B_r}   
	\rightarrow 0
	\ \mbox{ as } \ r \rightarrow \infty.
\end{equation*}
\end{lemma}
\begin{proof}
We follow the proof of \cite[Lemma 6.2]{Maehlen}.
Let $\varphi_r(x)=\varphi(x/r)$.
Using Lemma \ref{uniform-cont}, we find that for $f,g\in H^{s/2}(\mathbb{R})$
\begin{align*}
\abs{\langle A_rf,g\rangle}
&=
\frac 1{2\pi}
\left|
	\int_\mathbb{R}\int_\mathbb{R}\widehat{\varphi_r}(\eta)
	\widehat{f}(\xi-\eta)(m(\xi)-m(\xi-\eta))\overline{\widehat{g}(\xi)}
	\ \mathrm{d}\eta\ \mathrm{d}\xi
\right|
\\
&\lesssim
\int_\mathbb{R}
\left|
\widehat{\varphi_r}(\eta)\omega(\eta)
\right|
\int_\mathbb{R}\langle\xi-\eta\rangle^\frac{s}{2}\abs{\widehat{f}(\xi-\eta)}\langle\eta\rangle^\frac{s}{2}\abs{\widehat{g}(\eta)}\ \mathrm{d}\eta\mathrm{d}\xi\\
&\lesssim \int_\mathbb{R}\abs{\widehat{\varphi}(\eta)\omega\left(\eta/r\right)}\ \mathrm{d}\eta\norm{f}_{H^{s/2}}\norm{g}_{H^{s/2}}.
\end{align*}
Hence $\norm{A_r} \lesssim \int_\mathbb{R}\abs{\widehat{\varphi}(\eta)\omega(\eta/r)}\ \mathrm{d}\eta$
and this last integral tends to zero by
the dominated convergence theorem as $r\rightarrow \infty$,
since $\omega$ is bounded above by a polynomial and
$\lim_{\eta\rightarrow 0}\omega(\eta)\rightarrow 0$.

Similarly, for $f,g\in L^2(\mathbb{R})$ we have
\begin{multline*}
%
\abs{\langle B_rf, g\rangle}
=
\frac 1{2\pi}
\left|
\int_\mathbb{R}\int_\mathbb{R}\widehat{\varphi_r}(\eta)\widehat{f}(\xi-\eta)(m^{-1/2}(\xi)-m^{-1/2}(\xi-\eta))\overline{\widehat{g}(\xi)}\ \mathrm{d}\eta\mathrm{d}\xi
\right|
\\
=
\frac 1{2\pi}
\left|
\int_\mathbb{R}\int_\mathbb{R}\widehat{\varphi_r}(\eta)\widehat{f}(\xi-\eta)\left(\frac{m(\xi-\eta)-m(\xi)}{m^{1/2}(\xi-\eta)m^{1/2}(\xi)(m^{1/2}(\xi-\eta)+m^{1/2}(\xi))}\right)\overline{\widehat{g}(\xi)}\ \mathrm{d}\eta\mathrm{d}\xi
\right|
\\
\lesssim
\int_\mathbb{R} \abs{\widehat{\varphi_r}(\eta)\omega(\eta)}\int_\mathbb{R}\frac{\langle\xi-\eta\rangle^{s/2}}{m^{1/2}(\xi-\eta)}\abs{\widehat{f}(\xi-\eta)}\frac{\langle\eta\rangle^{s/2}}{m^{1/2}(\eta)}\abs{\widehat{g}(\eta)}\ \mathrm{d}\eta\mathrm{d}\xi
\\
\lesssim
\int_\mathbb{R}\abs{\widehat{\varphi}(\eta)\omega\left(\eta/r\right)}\ \mathrm{d}\eta\norm{f}_{L^2}\norm{g}_{L^2},
%
\end{multline*}
and we can conclude in the same way as before that
$\norm{B_r} \rightarrow 0$ as $r\rightarrow \infty$.
\end{proof}

\section{Near minimizers}
\label{near-min_section}

In this section we provide necessary estimates for
the infimum
\begin{equation}
\label{inf_definition}
	I_q = \inf _{u \in U_q} \mathcal{E}(u)
\end{equation}
and for those $u \in U_q$ that give values $\mathcal E(u)$
close to this infimum.
The regarded functional \eqref{general_E_definition}
is non-negative and so the same is true for the infimum.
However, we also need an upper bound for $I_q$
and this is addressed in the next result. 
\begin{proposition}\label{near-min}
There exist constants $D, q_0>0$ such that for $q\in (0,q_0)$
\begin{equation*}
	0 \leqslant I_q < m(0)q - Dq^{1+\beta},
\end{equation*}
with $\beta=\frac{s'}{2s'-1}$.
\end{proposition}
\begin{proof}
It is immediate that $0 \leqslant I_q$.
To establish the other inequality we consider $\varphi\in C^\infty(\mathbb{R})$, with $\text{supp}(\hat{\varphi})\subseteq(-1,1)$, $\varphi(x)\leq 0$, $x\in \mathbb{R}$ and $\mathcal{Q}(\varphi)=1$. We rescale and define $\varphi_{q,\alpha}(x)=\sqrt{q/\alpha}\varphi(x/\alpha)$, $\alpha>1$, so that $\mathcal{Q}(\varphi_{q,\alpha})=q$.

We first note that
\begin{equation}\label{Lphi-est}
\mathcal{L}(\varphi_{q,\alpha})\leq m(0)q+C_1q\alpha^{-s'},\ C_1>0,
\end{equation}
and using Proposition \ref{N-est}
\begin{equation}\label{Nrphi-est}
\abs{\mathcal{N}_r(\varphi_{q,\alpha})}\leq C_2q^2,\ C_2>0.
\end{equation}
In order to estimate $\mathcal{N}_c(\phi_{q,\alpha})$ we begin by estimating
\begin{align*} 
0\leq m^{1/2}(\xi)-m^{1/2}(0)&\leq \frac{m(\xi)-m(0)}{2\sqrt{m(0)}},\\
\abs{m^{-1/2}(\xi)-m^{-1/2}(0)}&\leq \frac{m(\xi)-m(0)}{2m(0)\sqrt{m(0)}},
\end{align*}
and then, using Lemma \ref{N-decomp}, we find that
\begin{align*}
\abs{\mathcal{N}_{2c}(\varphi_{q,\alpha})}&\lesssim q^{3/2}\alpha^{-s'-1/2},\\
\abs{\mathcal{N}_{3c}(\varphi_{q,\alpha})}&\lesssim q^{3/2}\alpha^{-s'}.
\end{align*}
Moreover, since $\varphi(x)\leq 0$, we have that
\begin{align*}
\frac{1}{2\sqrt{m(0)}}\int_\mathbb{R}\varphi_{q,\alpha}(x)^3\ \mathrm{d}x&=-2C_0q^{3/2}\alpha^{-1/2},\ C_0>0\\
\mathcal{N}_{1c}(\varphi_{q,\alpha})&\leq 0.
\end{align*}
Hence, it follows from the above estimates that there exists $\alpha_0>1$, such that for $\alpha\geq \alpha_0$,
\begin{equation*}
\mathcal{N}_c(\varphi_{q,\alpha})\leq -C_0q^{3/2}\alpha^{-1/2},
\end{equation*}
and combining this with \eqref{Lphi-est}, \eqref{Nrphi-est}, yields
\begin{equation}\label{Ephi-est}
\mathcal{E}(\varphi_{q,\alpha})\leq m(0)q-\left(C_0q^{3/2}\alpha^{-1/2}-C_1q\alpha^{-s'}\right)+C_2q^2,
\end{equation}
and by choosing $\alpha^{-s'}=Bq^\beta$, with $0<B\leq \alpha_0^{-s'}q^{-\beta}$, so that $\alpha\geq \alpha_0$, we get from \eqref{Ephi-est} that
\begin{equation}
\mathcal{E}(\varphi_{q,\alpha})\leq m(0)q-\underbrace{(C_0B^{1/(2s')}-C_1B)}_{=:2D}q^{1+\beta}+C_2q^2,
\end{equation}
By choosing $B$ small enough we have that $D>0$, and if we in addition choose $q_0$ sufficiently small, we find that
\begin{equation*}
I_q\leq \mathcal{E}(\varphi_{q,\alpha})< m(0)q-Dq^{1+\beta}.
\end{equation*}
\end{proof}

Note that it is possible to weaken the restriction
on $s'$ in the last proposition, namely
imposing $s' > 1/2$ instead.
It can be done by more accurate estimate of
the term $\abs{\mathcal{N}_r(\varphi_{q,\alpha})}$.
In other words it is not important here that $\beta < 1$.
However, it will be important for the subadditivity of $I_q$
below.
%
%
%
%

We now define a near minimizer to be an element $u$ of $U_{q}$ such that
\begin{equation}\label{near-min-ineq}
	\mathcal E(u) < m(0)q - Dq^{1+\beta}.
\end{equation}
By the previous proposition, there exist such elements $u \in U_{q}$. 
\begin{proposition}\label{near-min-est-prop}
A near minimizer $u \in U_{q}$ satisfies
\begin{equation*}
\norm{u}_{H^{s/2}}^2 \lesssim q.
\end{equation*}
\end{proposition}
\begin{proof}
Using propositions  \ref{L-est}, \ref{N-est} and \ref{near-min}, we find that 
\begin{align*}
\norm{u}_{H^\frac{s}{2}(\mathbb{R})}^2&\simeq \mathcal{L}(u)\\
&=\mathcal{E}(u)-\mathcal{N}(u)\\
&\lesssim m(0)q-Dq^{1+\beta}+\norm{u}_{L^2(\mathbb{R})}^2\norm{u}_{H^\frac{s}{2}(\mathbb{R})}+\norm{u}_{L^2(\mathbb{R})}^4\\
&\lesssim m(0)q-Dq^{1+\beta}+q\norm{u}_{H^\frac{s}{2}(\mathbb{R})}+q^2.
\end{align*}
Hence, it follows that for $q$ sufficiently small
\begin{equation*}
\norm{u}_{H^\frac{s}{2}(\mathbb{R})}^2\lesssim m(0)q-Dq^{1+\beta}\lesssim q.
\end{equation*}
\end{proof}
We next show that $I_q$ is strictly subadditive as a function of $q$.
This is essential when proving that dichotomy cannot occur. 
\begin{proposition}\label{subadditivity}
For any $q_1,q_2\in(0,q_0)$ such that $q_1+q_2\in(0,q_0)$,
holds
\begin{equation}\label{suba}
0<I_{q_1+q_2}<I_{q_1}+I_{q_2}.
\end{equation}
\end{proposition}
\begin{proof}
We show that $I_q$ is strictly subhomogeneous, i.e
\begin{equation}\label{subh}
I_{aq}<aI_q, \ a>1, q<aq<q_0,
\end{equation}
from which the strict subadditivity follows from a standard argument.
First we show that \eqref{subh} holds for $a\in(1,2]$.
Let $\{u_n\}_{n=1}^\infty$ be a minimizing sequence. From \eqref{near-min-ineq} we have that
\begin{equation}\label{near-min-ineq-decomp}
\mathcal{L}(u_n)+\mathcal{N}_c(u_n)+\mathcal{N}_r(u_n)<m(0)q-Dq^{1+\beta},
\end{equation}
and since $\mathcal{L}(u_n)\geq m(0)q$, $\mathcal{N}_r(u)\geq 0$, we get from \eqref{near-min-ineq-decomp} that
\begin{equation}\label{N_c-est}
\mathcal{N}_c(u)<-Dq^{1+\beta}.
\end{equation}
We also note that $\sqrt{a}-1\geq (a-1)/(1+\sqrt{2})$. With this in mind we see that
\begin{align*}
I_{aq}&\leq \mathcal{E}(a^{1/2}u_n)\\
&=\mathcal{L}(a^{1/2}u_n)+\mathcal{N}(a^{1/2}u_n)\\
&=a\mathcal{L}(u_n)+a^{3/2}\mathcal{N}_c(u_n)+a^2\mathcal{N}_r(u_n)\\
&=a\mathcal{E}(u_n)-a(\mathcal{N}_c(u_n)+\mathcal{N}_r(u_n))+a^{3/2}\mathcal{N}_c(u_n)+a^2\mathcal{N}_r(u_n)\\
&=a\mathcal{E}(u_n)+(a^{3/2}-a)\mathcal{N}_c(u_n)+(a^2-a)\mathcal{N}_r(u_n)\\
&\leq a\mathcal{E}(u_n)-(a^{3/2}-a)Dq^{1+\beta}+(a^2-a)C_3q^2\\
&\leq  a\mathcal{E}(u_n)-(a^2-a)\left(\frac{Dq^{1+\beta}}{1+\sqrt{2}}-C_3q^2\right).
\end{align*}
Hence, for $q_0$  sufficiently small
\begin{equation*}
I_{aq}+(a^2-a)\frac{Dq^{1+\beta}}{2\sqrt{2}}<aI_q,
\end{equation*}
which implies \eqref{subh} for $a\in(1,2]$, but also that $I_q>0$, for $q\in(0,q_0)$, proving the first inequality in \eqref{suba}. For the general case when $a>1$, we choose $l\in\mathbb{N}$ sufficiently big so that $a\in(1,2^l]$. Then $a^{1/l}\in(1,2]$, and so
\begin{equation*}
I_{aq}=I_{a^{1/l}a^{(l-1)/l}q}<a^{1/l}I_{a^{(l-1)/l}q}=a^{1/l}I_{a^{1/l}a^{(l-2)/l}q}<a^{2/l}I_{a^{(l-2)/l}q}<\ldots<aI_q.
\end{equation*}
\end{proof}

\section{Existence of minimizers}\label{ex-min}

In order to establish the existence of minimizers, we will apply the concentration-compactness principle (Theorem \ref{T.concentration-compactness}) to $e_n=u_n^2$, where $\{u_n\}_{n=1}^\infty$ is a minimizing sequence. The idea is to show that the vanishing and dichotomy scenarios cannot occur and then prove the existence of a minimizer using concentration. We start by excluding the vanishing scenario.
\begin{proposition}\label{vanishing-prop}
Vanishing does not  occur. 
\end{proposition}

\begin{proof}

	Let $\{u_n\}_{n=1}^\infty\subseteq U_q$ be a minimizing sequence of $\mathcal{E}$. We point out here that since $L^{-1/2}u_n\in H^s(\mathbb{R})$ and $s>1/2$, we have that $L^{-1/2}u_n\in L^{\infty}(\mathbb{R})$.
%
	By Lemma \ref{Nc-estimate} we have
	\[
		| \mathcal N_c(u) | \lesssim
		\lVert u \rVert_{L^2 (\mathbb R)}^2
		\lVert L^{-1/2}u \rVert_{L^{\infty} (\mathbb R)}
	\]
	and so for a minimizing sequence
	\[
		q^{\beta} \lesssim
		\lVert L^{-1/2}u_n \rVert_{L^{\infty} (\mathbb R)}
		,
	\]
where we used \eqref{N_c-est}.
	Arguing as in the proof of \cite[Lemma 4.5]{Duchene_Nilsson_Wahlen}, we have for any $x \in \mathbb R$ that
	\begin{multline*}
		\lVert L^{-1/2}u_n \rVert_{L^{\infty} (x - 1, x + 1)}
		\lesssim
		\lVert L^{-1/2}u_n \rVert_{L^2 (x - 1, x + 1)}
		^{1 - 1/(2s)}
		\lVert L^{-1/2}u_n \rVert_{H^s (\mathbb R)}
		^{1/(2s)}
		\\
		\lesssim
		\lVert L^{-1/2}u_n \rVert_{L^2 (x - 1, x + 1)}
		^{1 - 1/(2s)}
		\lVert u_n \rVert_{H^{\frac{s}{2}} (\mathbb R)}
		^{1/(2s)}
		\lesssim
		q^{1/(4s)}
		\lVert L^{-1/2}u_n \rVert_{L^2 (x - 1, x + 1)}
		^{1 - 1/(2s)},
	\end{multline*}
	and hence
	\[
		q^{\beta - 1/(4s)}
		\lesssim
		\sup_{x \in \mathbb R}
		\lVert L^{-1/2}u_n \rVert_{L^2 (x - 1, x + 1)}
		^{1 - 1/(2s)},
	\]
	which means that $L^{-1/2}u_n$ cannot vanish.
	Now we  show that $L^{-1/2}u_n$ is vanishing if
	one assumes that $u_n$ is vanishing.
	In order to do this we start by decomposing
\begin{align*}
(L^{-1/2}u_n)(x)&=(\mathcal{F}^{-1}(m^{-1/2}) * u_n)(x)\\
&=\int_\mathbb{R}\mathcal{F}^{-1}(m^{-1/2})(y)u_n(x-y)\ \mathrm{d}y\\
&=\underbrace{\int_{\abs{y}<\epsilon}\mathcal{F}^{-1}(m^{-1/2})(y)u_n(x-y)\ \mathrm{d}y}_{=:I_1}+\underbrace{\int_{\epsilon\leq \abs{y}\leq R}\mathcal{F}^{-1}(m^{-1/2})(y)u_n(x-y)\ \mathrm{d}y}_{=:I_2}\\
&\quad+\underbrace{\int_{\abs{y}\geq R}\mathcal{F}^{-1}(m^{-1/2})(y)u_n(x-y)\ \mathrm{d}y}_{=:I_3},
\end{align*}
and so
\begin{equation*}
\norm{L^{-1/2}u_n}_{L^2(\tilde{x}-1,\tilde{x}+1)}\leq \norm{I_1}_{L^2(\tilde{x}-1,\tilde{x}+1)}+\norm{I_2}_{L^2(\tilde{x}-1,\tilde{x}+1)}+\norm{I_3}_{L^2(\tilde{x}-1,\tilde{x}+1)}.
\end{equation*}
The goal is then to show that each of the above integrals can be made arbitrarily small.

By assumption there exists
\(
	p \in (1, 2) \cap [2 / (s+1), 2)
\)
such that \eqref{int-cond2} holds, and so\newline
\(
	\left\lVert
		\mathcal F^{-1} \left(m^{-1/2}\right)
	\right\rVert _{L^p(- \varepsilon, \varepsilon)}
	= o(1)
\)
as $\varepsilon \to 0$.
On the other hand its dual number $p'$ satisfies
condition
\(
	1/2 - 1/p' \leqslant s/2
\)
resulting in the embedding
\(
	H^{\frac{s}{2}}(\mathbb R) \hookrightarrow L^{p'}(\mathbb R)
	.
\)
Thus applying H\"older's inequality to $I_1$ yields
\[
	\norm{I_1}_{L^2(\tilde{x}-1,\tilde{x}+1)}^2
	\leqslant
	\int _{\tilde{x}-1}^{\tilde{x}+1}
	\left\lVert
		\mathcal F^{-1} \left(m^{- 1/2}\right)
	\right\rVert _{L^p(- \varepsilon, \varepsilon)}^2
	\lVert u_n \rVert _{L^{p'}(\mathbb R)}^2
	\ \mathrm{d}x
	= o(1)
	\mbox{ as } \varepsilon \to 0.
\]
For  $I_3$ we apply the Cauchy–Schwarz inequality
as follows
\[
	\norm{I_3}_{L^2(\tilde{x}-1,\tilde{x}+1)}^2
	\leqslant
	\int _{\tilde{x}-1}^{\tilde{x}+1}
	\left\lVert
		\mathcal F^{-1} \left(m^{- 1/2}\right)
	\right\rVert _{L^2( \mathbb R \setminus (- R, R) )}^2
	\lVert u_n \rVert _{L^2(\mathbb R)}^2
	\ \mathrm{d}x
	= o(1)
	\mbox{ as } R \to \infty.
\]
After choosing $\varepsilon$, $R$ we turn our attention
to $I_2$
\begin{multline*}
	\norm{I_2}_{L^2(\tilde{x}-1,\tilde{x}+1)}^2
	\leqslant
	\int _{\tilde{x}-1}^{\tilde{x}+1}
	\left\lVert
		\mathcal F^{-1} \left(m^{- 1/2}\right)
	\right\rVert _{L^2( (- R, R) \setminus (- \varepsilon, \varepsilon) )}^2
	\lVert u_n(x - y) \rVert _{L^2( \varepsilon < |y| < R )}^2
	\ \mathrm{d}x
	\\
	\leqslant
	C(\varepsilon, R)
	\lVert u_n \rVert _{L^2( -1 - |\tilde{x}| - R , 1 + |\tilde{x}| + R )}^2
	\to 0
	\mbox{ as } n \to \infty,
\end{multline*}
if one assumes vanishing of $u_n$.
\end{proof}

We next turn our attention to the dichotomy scenario.
\begin{proposition}\label{D-prop}
Dichotomy cannot occur.
\end{proposition}
\begin{proof}
Let $\chi :\mathbb{R}\rightarrow [0,1]$ be a smooth cutoff function with $\chi(x)=1$, for $\abs{x}\leq 1$ and $\chi(x)=0$, for $\abs{x}\geq 2$, and such that
\begin{equation*}
\chi=\chi_1^2,\quad 1-\chi=\chi_2^2,
\end{equation*}
where $\chi_1,\chi_2$ are smooth. Next, let $w_n(x)=u_n(x-x_n)$ and 
\begin{equation*}
w_n^{(1)}(x)=\underbrace{\chi_1\left(\frac{x}{M_n}\right)}_{=:\chi_{1n}(x)}w_n(x),\quad w_n^{(2)}(x)=\underbrace{\chi_2\left(\frac{x}{M_n}\right)}_{=:\chi_{2n}(x)}w_n(x),
\end{equation*}
Note that from the dichotomy assumption
\begin{align*}
\frac{1}{2}\int_{M_n\leq \abs{x}\leq 2M_n}w_n^2\ \mathrm{d}x&  \leq \frac{1}{2}\int_{M_n\leq \abs{x}\leq N_n}w_n^2\ \mathrm{d}x\\
&=\frac{1}{2}\int_{-N_n}^{N_n}w_n^2\ \mathrm{d}x-\frac{1}{2}\int_{-M_n}^{M_n}w_n^2\ \mathrm{d}x\\
&\rightarrow q^*-q^*\\
&=0.
\end{align*}
Since $\abs{w_n^{i}(x)}\leq \abs{w_n(x)}$, $i=1,2$, it follows directly that $\int_{M_n\leq \abs{x}\leq 2M_n}(w_n^{(i)})^2\ \mathrm{d}x\rightarrow 0$, as $n\rightarrow \infty$.
From this we can then deduce
\begin{equation*}
\frac{1}{2}\int_\mathbb{R}(w_n^{(1)})^2\ \mathrm{d}x=\frac{1}{2}\int_{-M_n}^{M_n}w_n^2\ \mathrm{d}x-\frac{1}{2}\int_{M_n\leq \abs{x}\leq 2M_n}(w_n^{(1)})^2\ \mathrm{d}x\rightarrow q^*,
\end{equation*}
and similarly 
\begin{equation*}
\frac{1}{2}\int_{\mathbb{R}}(w_n^{(2)})^2\ \mathrm{d}x=\frac{1}{2}\int_\mathbb{R}w_n^2 \mathrm{d}x-\frac{1}{2}\int_{-2M_n}^{2M_n}w_n^2\ \mathrm{d}x+\frac{1}{2}\int_{M_n\leq \abs{x}\leq 2M_n}(w_n^{(2)})^2\ \mathrm{d}x\rightarrow q-q^*.
\end{equation*}
We next show that
\begin{equation}\label{dichotomy-limit-main}
\mathcal{E}(w_n^{(1)})+\mathcal{E}(w_n^{(2)})-\mathcal{E}(w_n)\rightarrow 0,\ n\rightarrow \infty.
\end{equation}
As a first step towards this, we show that
\begin{equation}\label{dichotomy-limit-l}
\mathcal{L}(w_n^{(1)})+\mathcal{L}(w_n^{(2)})-\mathcal{L}(w_n)\rightarrow 0, \ n\rightarrow \infty
\end{equation}
Indeed, note that
\begin{equation*}
\mathcal{L}(w_n^{(1)})+\mathcal{L}(w_n^{(2)})-\mathcal{L}(w_n)=\frac{1}{2}\int_\mathbb{R}w_n^{(1)}Lw_n^{(1)}+w_n^{(2)}Lw_n^{(2)}-(\chi_{1n}^2+\chi_{2n}^2)w_nLw_n\ \mathrm{d}x,
\end{equation*}
and using Lemma \ref{commutator} we find that
\begin{align*}
\int_\mathbb{R}w_n^{(1)}Lw_n^{(1)}-\chi_{1n}^2w_nLw_n\ \mathrm{d}x&=\int_\mathbb{R}\chi_{1n}w_n(L(\chi_{1n}w_n)-\chi_{1n}Lw_n)\ \mathrm{d}x\\
&=\int_\mathbb{R}\chi_{1n}w_n[L,\chi_{1n}]\ \mathrm{d}x\\
&\rightarrow 0, \ n\rightarrow \infty
\end{align*}
In the same way we find that
\begin{equation*}
\int_\mathbb{R}w_n^{(2)}Lw_n^{(2)}-\chi_{2n}^2w_nLw_n\ \mathrm{d}x=\int_\mathbb{R}\chi_{2n}w_n[L,\chi_{2n}-1]w_n\ \mathrm{d}x\rightarrow 0, \ n\rightarrow \infty,
\end{equation*}
hence, \eqref{dichotomy-limit-l} holds.
The next step is to show that
\begin{equation}\label{dichotomy-limit-n}
\mathcal{N}(w_n^{(1)})+\mathcal{N}(w_n^{(2)})-\mathcal{N}(w_n)\rightarrow 0, \ n\rightarrow \infty,
\end{equation}
and for this we use the decomposition $\mathcal{N}=\mathcal{N}_c+\mathcal{N}_r$, and show that
\begin{align}
\mathcal{N}_c(w_n^{(1)})+\mathcal{N}_c(w_n^{(2)})-\mathcal{N}_c(w_n)&\rightarrow 0 \ n\rightarrow \infty\label{dichotomy-limit-nc},\\
\mathcal{N}_r(w_n^{(1)})+\mathcal{N}_r(w_n^{(2)})-\mathcal{N}_r(w_n)&\rightarrow 0 \ n\rightarrow \infty\label{dichotomy-limit-nr}.
\end{align}
Starting with \eqref{dichotomy-limit-nc}, we note that
\begin{align*}
\mathcal{N}_c(w_n^{(1)})+\mathcal{N}_c(w_n^{(2)})-\mathcal{N}_c(w_n)&=\frac{1}{2}\int_\mathbb{R}\bigg(L^{1/2}w_n^{(1)}(L^{-1/2}w_n^{(1)})^2+L^{1/2}w_n^{(2)}(L^{-1/2}w_n^{(2)})^2\\
&\qquad-(\chi_{1n}^2+\chi_{2n}^2)L^{1/2}w_n(L^{-1/2}w_n)^2\bigg)\ \mathrm{d}x,
\end{align*}
furthermore
\begin{align*}
&\int_\mathbb{R}L^{1/2}w_n^{(1)}(L^{-1/2}w_n^{(1)})^2-\chi_{1n}^2L^{1/2}w_n(L^{-1/2}w_n)^2\ \mathrm{d}x\\
&\quad=\int_\mathbb{R}L^{1/2}w_n^{(1)}(L^{-1/2}w_n^{(1)})^2-\chi_{1n}^2L^{1/2}w_n^{(1)}(L^{-1/2}w_n)^2\ \mathrm{d}x\\
&\qquad+\int_\mathbb{R}\chi_{1n}^2L^{1/2}w_n^{(1)}(L^{-1/2}w_n)^2-\chi_{1n}^2L^{1/2}w_n(L^{-1/2}w_n)^2\ \mathrm{d}x,
\end{align*}
and using Lemma \ref{commutator} we find that
\begin{align*}
&\int_\mathbb{R}L^\frac{1}{2}w_n^{(1)}(L^{-1/2}w_n^{(1)})^2-\chi_{1n}^2L^{1/2}w_n^{(1)}(L^{-1/2}w_n)^2\ \mathrm{d}x\\
&=\int_\mathbb{R}L^{1/2}w_n^{(1)}((L^{-1/2}w_n^{(1)})^2-\chi_{1n}^2(L^{-1/2}w_n)^2)\ \mathrm{d}x\\
&=\int_\mathbb{R}L^{1/2}w_n^{(1)}(L^{-1/2}w_n^{(1)}+\chi_{1n}L^{-1/2}w_n)(L^{-1/2}w_n^{(1)}-\chi_{1n}L^{-1/2}w_n)\ \mathrm{d}x\\
&=\int_\mathbb{R}L^{1/2}w_n^{(1)}(L^{-1/2}w_n^{(1)}+\chi_{1n}L^{-1/2}w_n)[L^{-1/2},\chi_{1n}]w_n\ \mathrm{d}x\\
&\rightarrow 0, \ n\rightarrow 0,
\end{align*}
and
\begin{align*}
&
\int_\mathbb{R}\chi_{1n}^2L^{1/2}w_n^{(1)}(L^{-1/2}w_n)^2-\chi_{1n}^2L^{1/2}w_n(L^{-1/2}w_n)^2\ \mathrm{d}x
\\
&=\int_\mathbb{R}\chi^2_{1n}L^{1/2}(w_n^{(1)}-w_n)(L^{-1/2}w_n)^2\ \mathrm{d}x\\
&=\int_\mathbb{R}L^{1/2}(\chi_{1n}(w_n^{(1)}-w_n))\chi_{1n}(L^{-1/2}w_n)^2\ \mathrm{d}x\\
&\quad-\int_\mathbb{R}[L^{1/2},\chi_{1n}](w_n^{(1)}-w_n)\chi_{1n}(L^{-1/2}w_n)^2\ \mathrm{d}x,
\end{align*}
where
\begin{align*}
&
\left|
\int_\mathbb{R}L^{1/2}(\chi_{1n}(w_n^{(1)}-w_n))\chi_{1n}(L^{-1/2}w_n)^2\ \mathrm{d}x
\right|
\\
&=
\left|
\int_\mathbb{R}\chi_{1n}(\chi_{1n}-1)w_nL^{1/2}(\chi_{1n}(L^{-1/2}w_n)^2)\ \mathrm{d}x
\right|
\\
&\leq \norm{\chi_{1n}(\chi_{1n}-1)w_n}_{L^2}\norm{L^{1/2}(\chi_{1n}(L^{-1/2}w_n)^2)}_{L^2}\\
&\lesssim \norm{w_n}_{L^2([-2M_n,-M_n]\cup[M_n,2M_n])}\norm{w_n}_{H^{s/2}}^2\\
&\rightarrow 0, \ n\rightarrow \infty,
\end{align*}
and $\int_\mathbb{R}[L^{1/2},\chi_{1n}](w_n^{(1)}-w_n)\chi_{1n}(L^{-1/2}w_n)^2\ \mathrm{d}x\rightarrow 0, \ n\rightarrow \infty$, according to Lemma \ref{commutator}.
Hence $\lim_{ n\rightarrow \infty}\int_\mathbb{R}L^{1/2}w_n^{(1)}(L^{-1/2}w_n^{(1)})^2-\chi_{1n}^2L^{1/2}w_n(L^{-1/2}w_n)^2\ \mathrm{d}x= 0,$ and in the same way we can show that $\lim_{ n\rightarrow \infty}\int_\mathbb{R}L^{1/2}w_n^{(2)}(L^{-1/2}w_n^{(2)})^2-\chi_{2n}^2L^{1/2}w_n(L^{-1/2}w_n)^2\ \mathrm{d}x= 0,$ which implies \eqref{dichotomy-limit-nc}. The limit \eqref{dichotomy-limit-nr} can be shown using similar techniques as \eqref{dichotomy-limit-nc} and we therefore omit the details.

We conclude that \eqref{dichotomy-limit-n} holds, which together with \eqref{dichotomy-limit-l} implies \eqref{dichotomy-limit-main}. Since $\{w_n\}_{n=1}^\infty$ is a minimizing sequence, we get that
\begin{equation}\label{lim1}
\lim_{n\rightarrow \infty}\mathcal{E}(w_n^{(1)})+\mathcal{E}(w_n^{(1)})\rightarrow I_q.
\end{equation}
However,
\begin{equation}\label{lim2}
\lim_{n\rightarrow \infty}
\left(
	\mathcal{E}\left(v_n^{(i)}\right)
	-
	\mathcal{E}\left(w_n^{(i)}\right)
\right)
=0, \ i=1,2,
\end{equation}
where $v_n^{(1)}=\sqrt{q^*/Q(w_n^{(1)})}w_n^{(1)}, \ v_n^{(2)}=\sqrt{(q-q^*)/Q(w_n^{(2)})}w_{n}^{(2)}$. By construction $v_n^{(1)}\in U_{q^*}$, $v_n^{(2)}\in U_{q-q^*}$, and so using \eqref{lim1}, \eqref{lim2}, we find that 
\begin{equation*}
I_q= \lim_{n\rightarrow \infty}\mathcal{E}(v_n^{(1)})+\mathcal{E}(v_n^{(1)})\geq I_{q^*}+I_{q-q^*},
\end{equation*}
which contradicts Proposition \ref{subadditivity}.
\end{proof}

\begin{proposition}\label{C-prop}
	There exists $u \in U_q$ solving minimization problem
	$\mathcal E(u) = I_q$.
\end{proposition}
\begin{proof}
By the concentration-compactness principle our
minimizing sequence
\(
	e_n = u_n^2 \in L^1 (\mathbb R)
	,
\)
$n \in \mathbb N$,
concentrates.
Moreover, due to the translation invariance
one can assume that it concentrates around zero,
and so
\[
	\int _{|x| > r} u_n^2(x)\ \mathrm{d}x \to  0
	\mbox{ uniformly with respect to }
	n \in \mathbb N
	\mbox{ as } r \to \infty
	.
\]
In addition, $\{u_n\}_{n=1}^\infty$ is a bounded sequence in $H^{\frac{s}{2}}(\mathbb{R})$ due to Proposition \ref{near-min-est-prop},
and so
\[
	\lVert (u_n)_h - u_n \rVert _{L^2}^2
	\lesssim
	q \left \lVert
		\xi \mapsto | e^{i\xi h} - 1 |
		\langle \xi \rangle ^{-s/2}
	\right \rVert
	_{L^{\infty}}^2
\]
that tends to zero uniformly with respect to
$ n \in \mathbb N $ as $h \to 0$.
Taking into account the boundedness of
$\{u_n\}_{n=1}^\infty$ in $L^2(\mathbb{R})$ one deduces from
the Frechet--Kolmogorov theorem
that $\{u_n\}_{n=1}^\infty$ is relatively compact in $L^2(\mathbb{R})$.
Thus we can assume that $\{u_n\}_{n=1}^\infty$ converges to some $u$ in $L^2(\mathbb{R})$.
Again using that $\{u_n\}_{n=1}^\infty$ is bounded in $H^{\frac{s}{2}}(\mathbb{R})$ , we may in addition assume that $u_n$
converges weakly in $H^{\frac{s}{2}}(\mathbb{R})$ to $u$.
Hence $u \in U_q$ and it is left to check
that it solves the minimization problem.

Firstly, applying the weak lower semi-continuity argument
we deduce
\[
	\mathcal L(u) \leqslant \liminf _{n \to \infty} \mathcal L(u_n).
\]
Indeed, the square root of $\mathcal L(u)$ defines a norm
in $H^{\frac{s}{2}}(\mathbb{R})$, equivalent to the standard Sobolev norm.
By the Mazur theorem a closed ball is weakly closed.
The latter property implies the weak lower semi-continuity
of the functional $\mathcal L$.

It is left to show that $\mathcal N(u_n)$
tends to $\mathcal N(u)$ as $n \to \infty$.
The cubic part is estimated as
\begin{multline*}
	\left| \mathcal N_c(u) - \mathcal N_c(u_n) \right|
	\leqslant
	\frac 12
	\left|
		\int_\mathbb{R} \left( L^{-1/2} u \right)^2 L^{1/2} (u- u_n)
	\ \mathrm{d}x\right|
	\\
	+
	\frac 12
	\left|
		\int_\mathbb{R}
		\left(
			\left( L^{-1/2} u \right)^2
			-
			\left( L^{-1/2} u_n \right)^2
		\right)		
		L^{1/2} u_n
	\ \mathrm{d}x\right|
	=
	\frac 12
	\left|
		\int_\mathbb{R} (u- u_n) L^{1/2} \left( L^{-1/2} u \right)^2
	\ \mathrm{d}x\right|
	\\
	+
	\frac 12
	\left|
		\int_\mathbb{R}
		\left(
			L^{-1/2} (u - u_n)
		\right)		
		\left(
			L^{-1/2} (u + u_n)
		\right)		
		L^{1/2} u_n
	\ \mathrm{d}x\right|
	\lesssim
	\left \lVert u - u_n \right \rVert _{L^2}
	\left \lVert
		L^{1/2} \left( L^{-1/2} u \right)^2
	\right \rVert _{L^2}
	\\
	+
	\left \lVert
		L^{-1/2} (u - u_n)
	\right \rVert _{H^{\frac{s}{2}}}
	\left \lVert
		L^{-1/2} (u + u_n)
	\right \rVert _{H^{\frac{s}{2}}}
	\left \lVert
		L^{1/2} u_n
	\right \rVert _{L^2}
	\lesssim q
	\left \lVert u - u_n \right \rVert _{L^2}
\end{multline*}
which tends to zero as $n \to \infty$.
For the remainder we have
\begin{multline*}
	\left| \mathcal N_r(u) - \mathcal N_r(u_n) \right|
	\\
	=
	\frac 18
	\left|
		\int_\mathbb{R}
		\left(
			L^{-1/2} (u - u_n)
		\right)		
		\left(
			L^{-1/2} (u + u_n)
		\right)		
		\left(
			\left( L^{-1/2} u \right)^2
			+
			\left( L^{-1/2} u_n \right)^2
		\right)		
	\ \mathrm{d}x\right|
	\\
	\lesssim
	\left \lVert
		L^{-1/2} (u - u_n)
	\right \rVert _{H^{s/2}}
	\left \lVert
		L^{-1/2} (u + u_n)
	\right \rVert _{H^{s/2}}
	\left(
		\left \lVert
			L^{-1/2} u
		\right \rVert _{L^4}^2
		+
		\left \lVert
			L^{-1/2} u_n
		\right \rVert _{L^4}^2
	\right)
	\\
	\lesssim q ^{3/2}
	\left \lVert u - u_n \right \rVert _{L^2}
\end{multline*}
that tends to zero as $n \to \infty$.
Summing up we obtain
\[
	I_q \leqslant \mathcal E(u)
	\leqslant \liminf _{n \to \infty}\mathcal{E}(u_n) = I_q
\]
which concludes the proof.
\end{proof}
We finish the proof of Theorem \ref{main-res-1} by proving the estimate. Let $u$ be a minimizer. We know that $u$ satisfies the Euler--Lagrange equation
\begin{equation*}
\lambda u+\mathrm{d}\mathcal{E}(u)=0.
\end{equation*}
Taking the inner product in this equation with $u$ yields
\begin{align}
-2\lambda q&=\mathrm{d}\mathcal{E}(u)(u)\nonumber\\
&=2\mathcal{L}(u)+3\mathcal{N}_c(u)+4\mathcal{N}_r(u)\nonumber\\
&=-\mathcal{L}(u)+3\mathcal{E}(u)+4\mathcal{N}_r(u)\label{lambda-est-calc}
\end{align}
Since $\mathcal{L}(u)\geq m(0) q$ and $\abs{\mathcal{N}_c(u)}=\mathcal{O}(q^{3/2})$, $\abs{\mathcal{N}_r(u)}=\mathcal{O}(q^2)$ by Proposition \ref{N-est}, it is easy to see from the second inequality in \eqref{lambda-est-calc} that for $q$ sufficiently small 
\begin{equation*}
-\lambda>\frac{m(0)}{2}.
\end{equation*}
For the upper bound we use \eqref{lambda-est-calc} together with propositions \ref{N-est}, \ref{near-min}, \ref{near-min-est-prop} to deduce that
\begin{align*}
-2\lambda q&=-\mathcal{L}(u)+3\mathcal{E}(u)+4\mathcal{N}_r(u)\\
&=-\mathcal{L}(u)+3I_q+4\mathcal{N}_r(u)\\
&\leq -m(0)q+3(m(0)q-Dq^{1+\beta})+\mathcal{O}(q^2)\\
&=2m(0)q-3Dq^{1+\beta}+\mathcal{O}(q^2),
\end{align*}
hence, for $q$ sufficiently small
\begin{equation*}
-\lambda<m(0)-Dq^\beta.
\end{equation*}

\section{Long wave approximation}
%
%
In this section we return to the initial variational
problem for the Whitham--Boussinesq system.
So from now on $L = K$ defined by \eqref{K_definition}.
We point out that all calculations below are also
valid for $L = 1 - \partial_x^2 / 3$ as well.
We will show that all minimizers are infinitely smooth
and refine existing estimates for them.

\begin{lemma}
\label{higher-reg-lemma}
	There exists $q_0 > 0$ such that for each $r \geqslant 0$
	holds
	\(
		\lVert u \rVert _{H^r}^2 \lesssim q
	\)
	uniformly for $q \in (0, q_0)$ and $u \in D_q$.
\end{lemma}
\begin{proof}
Firstly, one can notice that the statement holds for $r \in [0, 1/2]$, due to Proposition \ref{near-min-est-prop}.
We will extend the result by induction to bigger values of $r$
applying Formula \eqref{maineq1}.

Let $r \geqslant 1/2,$ then from the equivalence
of operators $K$, $J$ and product estimates in Sobolev spaces
we deduce
\[
	\left \lVert
		K^{-1/2} \left( K^{-1/2} v \right) ^3
	\right \rVert  _{H^r}
	\lesssim
	\lVert v \rVert  _{H^r}^3
	,
\]
\[
	\left \lVert
		K^{-1/2} \left( (K^{1/2} v)( K^{-1/2} v) \right)
	\right \rVert  _{H^r}
	\lesssim
	\lVert v \rVert  _{H^r}^2
	,
\]
\[
	\left \lVert
		K^{1/2} \left( K^{-1/2} v \right) ^2
	\right \rVert  _{H^r}
	\lesssim
	\lVert v \rVert  _{H^r}^2
\]
for any $v \in H^r(\mathbb{R})$.
All three constants here depend only on $r$.

Now for any minimizer $u \in D_q$
calculate $Ku$ by Formula \eqref{maineq1} and
obtain
\begin{multline*}
	\lVert u \rVert  _{H^{r + 1}}
	\lesssim	
	\lVert Ku \rVert  _{H^r}
	\leqslant
	| \lambda | \lVert u \rVert  _{H^r}
	+ \frac 12
	\left \lVert
		K^{-1/2} \left( K^{-1/2} u \right) ^3
	\right \rVert  _{H^r}
	\\
	+
	\left \lVert
		K^{-1/2} \left(( K^{1/2} u)( K^{-1/2} u) \right)
	\right \rVert  _{H^r}
	+ \frac 12
	\left \lVert
		K^{1/2} \left( K^{-1/2} u \right) ^2
	\right \rVert  _{H^r}
	\lesssim
	\sqrt{q}
\end{multline*}
for any $r \geqslant 1/2$.
We have used $|\lambda| \leqslant 1$
according to Theorem \ref{main-res-1}.
This concludes the proof by induction.
\end{proof}

\begin{lemma}\label{higher-regularity-est-lemma}
	There exist $q_0 > 0$ and $C > 0$ such that
	the following estimates hold
	\begin{equation}
	\label{refined_estimate1}
		\lVert u \rVert _{L^{\infty}} \leqslant Cq^{2/3}
		,	
	\end{equation}
	\begin{equation}
	\label{refined_estimate2}
		\lVert \partial_x u \rVert _{L^2}^2 \leqslant Cq^{5/3}
		,	
	\end{equation}
	\begin{equation}
	\label{refined_estimate3}
		\lVert \partial_x^2 u \rVert _{L^2}^2 \leqslant Cq^{7/3}
	\end{equation}
	uniformly for $q \in (0, q_0)$ and $u \in D_q$.
\end{lemma}
\begin{proof}
Introducing the notation
\[
	M(u) = 
	\frac 12 K^{-1/2} \left( K^{-1/2} u \right) ^3
	+ K^{-1/2} \left(( K^{1/2} u)( K^{-1/2} u) \right)
	+ \frac 12 K^{1/2} \left( K^{-1/2} u \right) ^2
\]
one can rewrite Equation \eqref{maineq1} in the form
\[
	( \lambda + K ) u = - M(u)
	.
\]
Note that
\(
	- \lambda
	\in \left( 0, 1 - Dq^{2/3} \right)
\)
according to Theorem \ref{main-res-1}
and so $\lambda + 1 > Dq^{2/3}$.
The Fourier transform of minimizer $u$ can be estimated as
\[
	\left|
		\widehat{u}(\xi)
	\right|
	=
	\left|
		\frac{ \mathcal F( M(u) ) }{ \lambda + m(\xi) }
	\right|
	\leqslant
	\frac{ | \mathcal F( M(u) ) | }{ Dq^{2/3} + m(\xi) - 1 }
	\lesssim
	| \mathcal F( M(u) )(\xi) |
	\left(
		\frac{ \chi _{ |\xi| \leqslant 1 }(\xi) }{ q^{2/3} + \xi^2 }
		+	
		\frac{ \chi _{ |\xi| > 1 }(\xi) }{ q^{2/3} + |\xi| }
	\right)
\]
where $\chi_A(\xi)$
stands for the characteristic function of a set $A$.
As was shown in the proof of Lemma \ref{higher-reg-lemma}
$M(u)$, is smooth and its $H^s$-norm is bounded by $q$
for any non-negative $s$.
Hence $\mathcal F( M(u) )$ multiplied by any
power of $\xi$ is bounded by $q$
with respect to $L^2$-norm.

Let us show that the $L^{\infty}$-norm of
$\mathcal F( M(u) )$ is bounded by $q$.
Indeed, we have
\[
	\left|
		\mathcal F
		\left(
			K^{1/2} \left( K^{-1/2} u \right) ^2
		\right)
		(\xi)
	\right|
	\lesssim
	\int_\mathbb{R}
	\frac
	{ \sqrt{ m(\xi) } | \widehat{u}(\xi - \zeta) \widehat{u}(\zeta) | }
	{ \sqrt{ m(\xi - \zeta) m(\zeta) } }
	\ \mathrm{d}\zeta
	\lesssim
	\lVert u \rVert _{L^2}^2
	\lesssim
	q
	,
\]
\[
	\left|
		\mathcal F
		\left(
			K^{-1/2} \left(( K^{1/2} u)( K^{-1/2} u) \right)
		\right)
		(\xi)
	\right|
	\lesssim
	\int_\mathbb{R}
	\frac
	{ \sqrt{ m(\xi - \zeta) } | \widehat{u}(\xi - \zeta) \widehat{u}(\zeta) | }
	{ \sqrt{ m(\xi) m(\zeta) } }
	\ \mathrm{d}\zeta
	\lesssim
	\lVert u \rVert _{L^2}^2
	\lesssim
	q
\]
and similarly
\[
	\left|
		\mathcal F
		\left(
			K^{-1/2} \left( K^{-1/2} u \right) ^3
		\right)
		(\xi)
	\right|
	\lesssim
	\lVert u \rVert _{L^2}
	\left \lVert \left( K^{-1/2} u \right) ^2 \right \rVert _{L^2}
	\lesssim
	\lVert u \rVert _{L^2}^3
	\lesssim
	q^{3/2}
	.
\]
Thus
\(
	\lVert \mathcal F( M(u) ) \rVert _{L^{\infty}}
	\lesssim q
	.
\)
So we are in a position to prove \eqref{refined_estimate1},
indeed,
\begin{multline*}
	\lVert u \rVert _{L^{\infty}}
	\lesssim
	\lVert \widehat{u} \rVert _{L^1}
	\lesssim
	\int _{ |\xi| \leqslant 1 }
	\frac{ | \mathcal F( M(u) )(\xi) | }{ q^{2/3} + \xi^2 } d\xi
	+	
	\int _{ \abs{\xi}>1 }
	\frac{ | \mathcal F( M(u) )(\xi) |}{q^{2/3}+ |\xi| }
	 \ \mathrm{d}\xi
	\\
	\lesssim
	q^{-1/3}
	\lVert \mathcal F( M(u) ) \rVert _{L^{\infty}}
	+
	\lVert \mathcal F( M(u) ) \rVert _{L^2}
	\lesssim
	q^{2/3}
	.
\end{multline*}
Estimate \eqref{refined_estimate2} is proved as follows
\begin{multline*}
	\lVert \partial_x u \rVert _{L^2}^2
	=
	\lVert \xi \mapsto \xi \widehat{u}(\xi) \rVert _{L^2}^2
	\lesssim
	\int _{ |\xi| \leqslant 1 }
	\frac
	{ \xi^2 | \mathcal F( M(u) )(\xi) |^2 }
	{ ( q^{2/3} + \xi^2 )^2 } \ \mathrm{d}\xi
	+	
	\int _{ \abs{\xi}> 1 }
	\frac{ \xi^2| \mathcal F( M(u) )(\xi) |^2}{ (q^{2/3}+|\xi|)^2}
	 \ \mathrm{d}\xi
	\\
	\lesssim
	q^{-1/3}
	\lVert\mathcal F( M(u) ) \rVert _{L^{\infty}}^2
	+
	\lVert \mathcal F( M(u) ) \rVert _{L^2}^2
	\lesssim
	q^{5/3}
	.
\end{multline*}
A straightforward repetition of the last argument
for the second derivative of the minimizer
gives
\begin{multline}
\label{refined_estimate_proof_01}
	\lVert \partial_x^2 u \rVert _{L^2}^2
	=
	\lVert \xi \mapsto \xi^2 \widehat{u}(\xi) \rVert _{L^2}^2
	\lesssim
	\int _{ |\xi| \leqslant 1 }
	\frac
	{ \xi^4 | \mathcal F( M(u) )(\xi) |^2 }
	{ ( q^{2/3} + \xi^2 )^2 } \ \mathrm{d}\xi
	+	
	\int _{ \abs{\xi}>1 }
	\frac{\xi^4| \mathcal F( M(u) )(\xi) |^2 }{(q^{2/3}+ |\xi|)^2 }
	 \ \mathrm{d}\xi
	\\
	\lesssim
	q^{1/3}
	\lVert \mathcal F( M(u) ) \rVert _{L^{\infty}}^2
	+
	\lVert \mathcal F( \partial_x M(u) ) \rVert _{L^2}^2
	\lesssim
	q^{7/3}
	+
	\lVert \partial_x M(u) \rVert _{L^2}^2
\end{multline}
that is only $\mathcal O(q^2) $ and so weaker than \eqref{refined_estimate3}. 
However, Estimate \eqref{refined_estimate2}
is a refinement compared with Lemma \ref{higher-reg-lemma}, so it can be used for more delicate estimate
of the square norm
\(
	\lVert \partial_x M(u) \rVert _{L^2}^2
\)
as follows
\begin{multline*}
	\left \lVert
		\partial_x K^{1/2}
		\left( K^{-1/2} u \right) ^2
	\right \rVert _{L^2}
	\lesssim
	\norm
	{
		K^{-1/2} u  K^{-1/2} \partial_x u
	} _{H^{1/2}}
	\\
	\lesssim
	\norm
	{
		K^{-1/2} u
	} _{H^1}
	\norm
	{
		K^{-1/2} \partial_x u
	} _{H^{1/2}}
	\lesssim
	q^{4/3}
	,
\end{multline*}
%
where product estimates were used.
To continue, first note that the estimate of the
derivative \eqref{refined_estimate2}, will not be spoiled if one
changes $L^2$-norm to $H^s$-norm with any $s \geqslant 0$.
In other words,
\(
	\lVert \partial_x u \rVert _{H^s}
	\lesssim q^{5/6}
	,
\)
and so
\begin{multline*}
	\left \lVert
		\partial_x K^{-1/2} \left( K^{1/2} u K^{-1/2} u \right)
	\right \rVert _{L^2}
	\lesssim
	\norm
	{
		K^{1/2} \partial_x u  K^{-1/2} u
	} _{L^2}
	+
	\norm
	{
		K^{1/2} u  K^{-1/2} \partial_x u
	} _{L^2}
	\\
	\lesssim
	\left \lVert
		\partial_x u
	\right \rVert _{H^{3/2}}
	\left \lVert
		u
	\right \rVert _{L^2}
	+
	\left \lVert
		u
	\right \rVert _{H^{3/2}}
	\left \lVert
		\partial_x u
	\right \rVert _{L^2}
	\lesssim
	q^{4/3}
	.
\end{multline*}
The last remaining term is estimated similarly 
\[
	\left \lVert
		\partial_x K^{-1/2} \left( K^{-1/2} u \right) ^3
	\right \rVert _{L^2}
	\lesssim
	\left \lVert
			( K^{-1/2} u )^2 K^{-1/2} \partial_x u
	\right \rVert _{L^2}
	\lesssim
	\left \lVert
		u
	\right \rVert _{H^{1/2}} ^2
	\left \lVert
		\partial_x u
	\right \rVert _{L^2}
	\lesssim
	q^{11/6}
	.
\]
Thus
\[
	\lVert \partial_x M(u) \rVert _{L^2}
	\lesssim q^{4/3}
\]
that together with \eqref{refined_estimate_proof_01}
conclude the proof of Estimate \eqref{refined_estimate3}.
\end{proof}

\begin{remark}
	Lemmas \ref{higher-reg-lemma},
	\ref{higher-regularity-est-lemma}
	remain valid with the surface elevation
	$\eta_u$ and velocity $v_u$ defined by \eqref{eta-u}, \eqref{v-u}
	substituted instead of the minimizer $u \in D_q$.
\end{remark}

We now turn to the task of approximating the solutions found in Theorem \ref{main-res-1} with solutions of the KdV-equation. For this part we follow \cite{Duchene_Nilsson_Wahlen} closely.

We introduce the long-wave scaling $S_{\text{KdV}}(f)(x)=q^{2/3}f(q^{1/3}x)$ and note that when making the ansatz $u=S_{\text{KdV}}(\psi)$ in \eqref{maineq1}, the leading order part of the equation as $q\rightarrow 0$ is, with $\lambda=-1+\lambda_0q^{2/3}$,
\begin{equation}\label{kdv}
\lambda_0\psi+\frac{3}{2}\psi^2-\frac{\psi_{xx}}{3}=0.
\end{equation}
Equation \eqref{kdv} is the travelling wave version of the  KdV-equation, which has the up to translation the following unique solution
\begin{equation*}
\psi_{\text{KdV}}(x)
= -\lambda_0\sech^2\left(\frac{1}{2}\sqrt{3\lambda_0}x\right).
\end{equation*}
We note that \eqref{kdv} is the Euler-Lagrange equation of the minimization problem 
\begin{equation*}
I_{\text{KdV}}:=\min_{\psi\in V_1}\mathcal{E}_{KdV}(\psi),
\end{equation*}
where 
\begin{equation*}
\mathcal{E}_{KdV}(\psi):=\frac{1}{2}\int_\mathbb{R}\frac{\psi_x^2}{3}+\psi^3\ \mathrm{d}x,
\end{equation*}
and $V_1:=\{\psi\in H^1(\mathbb{R})\ :\ \mathcal{Q}(\psi)=1\}$. The constraint $\mathcal{Q}(\psi_{\text{KdV}})=1$ requires that $\lambda_0=3/16^\frac{1}{3}$.
The relation between $\mathcal{E}$ and $\mathcal{E}_{\text{KdV}}$ is now established.
\begin{lemma}\label{rel-kdv-lemma}
For $u\in H^2(\mathbb R)$ hold 
\begin{equation}\label{rel-kdv}
\mathcal{E}(u)=\mathcal{Q}(u)+\mathcal{E}_{\text{KdV}}(u)+\mathcal{E}_{\text{rem}}(u),
\end{equation}
with
\begin{align}\label{rem-est}
\abs{\mathcal{E}_{\text{rem}}(u)}&\lesssim \norm{\partial_x^2u}_{L^2}^2+\norm{u}_{L^\infty}\norm{\partial_xu}_{L^2}^2+\norm{u}_{L^2}^2\norm{\partial_xu}_{L^2}+\norm{u}_{L^4}^2\norm{\partial_x^2u}_{L^2}+\norm{u}_{L^2}^4,\\
\abs{\langle \mathrm{d}\mathcal{E}_{\text{rem}}(u),u\rangle}&\lesssim \norm{\partial_x^2u}_{L^2}^2+\norm{u}_{L^\infty}\norm{\partial_xu}_{L^2}^2+\norm{u}_{L^2}^2\norm{\partial_xu}_{L^2}+\norm{u}_{L^4}^2\norm{\partial_xu}_{L^2}+\norm{u}_{L^2}^4\label{diffrem-est}
\end{align}
\end{lemma}
\begin{proof}
We note that
\begin{align*}
\mathcal{E}(u)&=\frac{1}{2}\int_\mathbb{R}uKu\ \mathrm{d}x+\mathcal{N}_c(u)+\mathcal{N}_r(u)\\
&\quad =Q(u)+\frac{1}{2}\int_\mathbb{R}u(K-1)u\ \mathrm{d}x+\frac{1}{2}\int_\mathbb{R} u^3\ \mathrm{d}x+\mathcal{N}_{1c}(u)+\mathcal{N}_{2c}u+\mathcal{N}_{3c}(u)+\mathcal{N}_r(u)
\\
&
\quad=Q(u)+\mathcal{E}_{\text{KdV}}(u)
\\
&
\qquad+\underbrace{\frac{1}{2}\int_\mathbb{R}
\left(
	m(\xi)-1-\frac{\xi^2}{3}
\right) \abs{\hat{u}}^2\ \mathrm{d}\xi+\mathcal{N}_{1c}(u)+\mathcal{N}_{2c}(u)+\mathcal{N}_{3c}(u)+\mathcal{N}_r(u)}_{=:\mathcal{E}_{\text{rem}}(u)}
\end{align*}
Since $m(\xi)=\xi/\tanh(\xi)$, we have that $\abs{m(\xi)-1 - \frac{\xi^2}{3}}\lesssim \xi^4$, so that
\begin{equation*}
\int_\mathbb{R}
\left|
	m(\xi)- 1 - \frac{\xi^2}{3}
\right|
\abs{\hat{u}}^2\ \mathrm{d}\xi \lesssim \norm{\partial_x^2u}_{L^2}^2.
\end{equation*}
From Lemma \ref{N-decomp} we have
\begin{equation*}
\abs{\mathcal{N}_{1c}(u)}\lesssim \int_\mathbb{R}\abs{u((K^{-1/2}-1)u)^2 }\ \mathrm{d}x\lesssim \norm{u}_{L^\infty}\norm{\partial_xu}_{L^2}^2.
\end{equation*}
Similarly we find that
\begin{align*}
\abs{\mathcal{N}_{2c}(u)}&\lesssim \norm{u}_{L^4}^2\norm{\partial_xu}_{L^2},\\
\abs{\mathcal{N}_{3c}(u)}&\lesssim \norm{u}_{L^2}^2\norm{\partial_xu}_{L^2}.
\end{align*}
The term $\mathcal{N}_r(u)$ is estimated in Proposition \ref{N-est}, hence \eqref{rem-est} is established. The estimate \eqref{diffrem-est} is proved in a similar way and we therefore omit the details.
\end{proof}
\begin{lemma}
There exists $q_0>0$ such that 
\begin{align}
I_q&=q+\mathcal{E}_{\text{KdV}}(u)+\mathcal{O}(q^2),\ \text{uniformly over }u\in D_q,\label{min-rel-1}\\
I_q&=q+q^\frac{5}{3}I_{\text{KdV}}+\mathcal{O}(q^2).\label{min-rel-2}
\end{align}
\end{lemma}
\begin{proof}
Let $u\in D_q$. From Lemma \ref{higher-reg-lemma} we know that $u\in H^r(\mathbb{R})$ for any $r\geq 0$. In particular $u\in H^2(\mathbb{R})$, hence by Lemma  \ref{rel-kdv-lemma} 
\begin{equation*}
\mathcal{E}(u)=q+\mathcal{E}_{\text{KdV}}(u)+\mathcal{E}_{\text{rem}}(u).
\end{equation*}
Using \eqref{rem-est} together with Lemma \ref{higher-regularity-est-lemma}, we get $\abs{\mathcal{E}_{\text{rem}}(u)}\lesssim q^2$. Hence, \eqref{min-rel-1} follows.

Turning now to \eqref{min-rel-2} we let $\psi=S_{\text{KdV}}^{-1}(u)$ and note that $\psi\in V_1$ and
\begin{equation*}
\mathcal{E}_{\text{KdV}}(u)=q^{5/3}\mathcal{E}_{\text{KdV}}(\psi)\geq q^{5/3}I_{\text{KdV}},
\end{equation*}
so this together with \eqref{min-rel-1} implies 
\begin{equation*}
I_q\geq q+q^{5/3}I_{\text{KdV}}+\mathcal{O}(q^2).
\end{equation*}
On the other hand, $\tilde{u}:=S_{\text{KdV}}(\psi_{\text{KdV}})\in U_q$,
so again using \eqref{min-rel-1} obtain
\begin{align*}
I_q&\leq \mathcal{E}(\tilde{u})\\
&=q+\mathcal{E}_{\text{KdV}}(\tilde{u})+\mathcal{O}(q^2)\\
&=q+q^{5/3}\mathcal{E}_{\text{KdV}}(\psi_{\text{KdV}})+\mathcal{O}(q^2)\\
&=q+q^{5/3}I_{\text{KdV}}+\mathcal{O}(q^2),
\end{align*}
which concludes the proof of \eqref{min-rel-2}.
\end{proof}

The statement of Theorem \ref{main-res-2} is a summary of the following lemmas.
\begin{lemma}
\label{approx-thm}
There exists $q_0>0$ such that for any
$q\in (0,q_0)$ and $u\in D_q$ there exists $x_u\in \mathbb{R}$ such that 
\begin{equation*}
\norm{S^{-1}_{\text{KdV}}(u)-\psi_{\text{KdV}}(\cdot -x_u)}_{H^1}\lesssim q^{1/6},
\end{equation*}
uniformly with respect to $q\in(0,q_0)$ and $u\in D_q$.
\end{lemma}
The proof of Lemma \ref{approx-thm}
is identical to the proof of \cite[Theorem 5.5]{Duchene_Nilsson_Wahlen}
and is therefore omitted.
We next relate the two Lagrange multipliers $\lambda$ and $\lambda_0$.
\begin{lemma}
\label{lagrange-approx}
The Lagrange multipliers related to the minimization problem \eqref{general_min_problem}, satisfy 
\begin{equation*}
\lambda=-1+\lambda_0q^{2/3}+\mathcal{O}(q^{5/6}).
\end{equation*}
\end{lemma}

\begin{proof}
Let $u\in D_q$.
From Lemma \ref{rel-kdv-lemma} we have
\begin{equation}
\label{diff-eq-lemma}
\langle \mathrm{d}\mathcal{E}(u),u\rangle =2q+\langle\mathrm{d}\mathcal{E}_{\text{KdV}}(u),u\rangle +\mathcal{O}(q^2).
\end{equation}
Moreover, $\langle\mathrm{d}\mathcal{E}_{\text{KdV}}(u),u\rangle=q^{5/3}\langle \mathrm{d}\mathcal{E}_{\text{KdV}}(S_{\text{KdV}}^{-1}(u)),S_{\text{KdV}}^{-1}(u)\rangle$,
and by Lemmas \ref{higher-regularity-est-lemma}, \ref{approx-thm}
\begin{equation*}
\langle \mathrm{d}\mathcal{E}_{\text{KdV}}(S_{\text{KdV}}^{-1}(u)),S_{\text{KdV}}^{-1}(u)\rangle-\langle \mathrm{d}\mathcal{E}_{\text{KdV}}(\psi_{\text{KdV}}),\psi_{\text{KdV}}\rangle =\mathcal{O}(q^{1/6}).
\end{equation*}
Combining this with \eqref{diff-eq-lemma}, we obtain
\begin{equation}\label{diff-eq-lemma-2}
\langle \mathrm{d}\mathcal{E}(u),u\rangle=2q+q^{5/3}\langle \mathrm{d}\mathcal{E}_{\text{KdV}}(\psi_{\text{KdV}}),\psi_{\text{KdV}}\rangle+\mathcal{O}(q^{11/6}).
\end{equation}
On the other hand,
from the Euler-Lagrange equations we have
\begin{align*}
2\lambda q+\langle\mathrm{d}\mathcal{E}(u),u\rangle&=0,\\
2\lambda_0+\langle \mathrm{d}\mathcal{E}_{\text{KdV}}(\psi_{\text{KdV}}),\psi_{\text{KdV}}\rangle&=0,
\end{align*}
and when we combine this with \eqref{diff-eq-lemma-2}, we get
\begin{equation*}
-2\lambda q=2q-2\lambda_0q^{5/3}+\mathcal{O}(q^{11/6}),
\end{equation*}
and dividing with $-2q$ yields
\begin{equation*}
\lambda=-1+\lambda_0q^{2/3}+\mathcal{O}(q^{5/6}).
\end{equation*}
\end{proof}

For each solution $u$ of \eqref{maineq1},
we have the corresponding physical parameters $\eta_u$, $v_u$
defined by \eqref{eta-u}, \eqref{v-u}
where $-1/c^2=\lambda=-1+\lambda_0q^{2/3}+\mathcal{O}(q^{5/6})$
by Lemma \ref{lagrange-approx}.
We have the following estimates for $\eta_u$, $v_u$
that are similar to the one given in Lemma \ref{approx-thm}.

\begin{lemma}
There exists $q_0>0$ such that for $q\in (0,q_0)$ and $u\in D_q$ there exists $x_u\in \mathbb{R}$ such that 
\begin{equation*}
	\norm{
		S^{-1}_{\text{KdV}}(\eta_u)+\psi_{\text{KdV}}(\cdot -x_u)
	} _{H^{1/2}}
	\lesssim q^{1/6}
	,
\end{equation*}
\begin{equation*}
	\norm{
		S^{-1}_{\text{KdV}}(v_u)+\psi_{\text{KdV}}(\cdot -x_u)
	} _{H^{3/2}}
	\lesssim q^{1/6}
\end{equation*}
uniformly with respect to $q\in(0,q_0)$ and $u\in D_q$.
\end{lemma}
\begin{proof}

We will prove the first inequality.
The second one can be proved analogously.
Firstly, one can notice that due to
$1/2 < -\lambda < 1$
in accordance with to Estimate \eqref{crude-est-lambda},
it is enough to prove
\begin{equation}
\label{eta_estimate_proof_01}
	\norm
	{
		\lambda S^{-1}_{\text{KdV}}(\eta_u)
		+ \lambda \psi_{\text{KdV}}(\cdot -x_u)
	} _{H^{1/2}}\lesssim q^{1/6},
\end{equation}
where $x_u$ is taken as in Lemma \ref{approx-thm}.
The first term under the norm in \eqref{eta_estimate_proof_01}
has the form
\begin{equation*}
	\lambda S_{\text{KdV}}^{-1}(\eta_u)
	=
	\frac{q^{-2/3}}{2}\left(K^{-1/2}u\right)^2(q^{-1/3}\cdot)
	+
	q^{-2/3}(K^{1/2}u)(q^{-1/3}\cdot)
\end{equation*}
where the first element of the sum is negligible
in view of the straightforward estimate
\begin{equation*} 
	\norm{(K^{-1/2}u)^2(q^{-1/3}\cdot)}_{H^{1/2}}\lesssim q.
\end{equation*}
The second element of the sum can be rewritten as follows.
We note that
\begin{equation*}
(K^{1/2}u)(q^{-1/3}x)=(K_q^{1/2}u(q^{-1/3}\cdot))(x),
\end{equation*}
where we used $K_q$ to denote the Fourier multiplier operator with symbol $m(q^{1/3}\cdot)$.
We then get that $q^{-2/3}(K^{1/2}u)(q^{-1/3}\cdot)=K_q^{1/2}S_{\text{KdV}}^{-1}(u)$. Using this, we find that
\begin{multline*}
q^{-2/3}(K^{1/2}u)(q^{-1/3}\cdot)-\psi_{\text{KdV}}(\cdot-x_u)
=
K_q^{1/2}S_{\text{KdV}}^{-1}(u)-\psi_{\text{KdV}}(\cdot-x_u)
\\
=
K_q^{1/2}(S_{\text{KdV}}^{-1}(u)-\psi_{\text{KdV}}(\cdot-x_u))+(K_q^{1/2}-1)\psi_{\text{KdV}}(\cdot-x_u)
.
\end{multline*}
Here the last term is estimated as
\begin{equation*}
\norm{(K_q^{1/2}-1)\psi_{\text{KdV}}(\cdot-x_u)}_{H^{1/2}}\lesssim q^{1/3}\norm{\psi_{\text{KdV}}}_{H^{3/2}}.
\end{equation*}
Finally, we have
\begin{multline*}
	\norm
	{
		\lambda S^{-1}_{\text{KdV}}(\eta_u)
		+ \lambda \psi_{\text{KdV}}(\cdot -x_u)
	} _{H^{1/2}}
	\lesssim
	\norm
	{
		K_q^{1/2}(S_{\text{KdV}}^{-1}
		- \psi_{\text{KdV}}(\cdot-x_u))
	} _{H^{1/2}}
	\\
	+
	\norm
	{
		(K_q^{1/2}-1)\psi_{\text{KdV}}(\cdot-x_u)
	} _{H^{1/2}}
	+
	\norm
	{
		( 1 + \lambda ) \psi_{\text{KdV}}(\cdot-x_u)
	} _{H^{1/2}}
	+ q^{1/3}
\end{multline*}
that gives \eqref{eta_estimate_proof_01} by
Lemma \ref{approx-thm} and \ref{lagrange-approx}.

\end{proof}


\vskip 0.05in
\noindent
{\bf Acknowledgments.}
{
E.D. was supported by the
Norwegian Research Council.
D.N. was supported by an ERCIM `Alain Bensoussan' Fellowship
and by grant no. 250070 from the Research Council of Norway.
}

\bibliographystyle{acm}
\bibliography{bibliography}

\end{document}